\documentclass[12pt,a4paper]{article}

\usepackage[margin=3.3cm]{geometry}
\usepackage[T1]{fontenc}
\usepackage[utf8]{inputenc}
\usepackage{lmodern}
\usepackage{amsmath,amssymb,amsthm,mathtools}
\usepackage{graphicx}
\usepackage{tikz}
\usetikzlibrary{arrows.meta,calc}
\usepackage{float}
\usepackage{enumitem}
\usepackage{xcolor}
\usepackage[hidelinks]{hyperref}
\usepackage[nameinlink,noabbrev]{cleveref}
\usepackage{microtype}

\allowdisplaybreaks

\newtheorem{theorem}{Theorem}[section]
\newtheorem{proposition}[theorem]{Proposition}
\newtheorem{lemma}[theorem]{Lemma}
\newtheorem{corollary}[theorem]{Corollary}
\newtheorem{assumption}[theorem]{Assumption}
\theoremstyle{definition}
\newtheorem{definition}[theorem]{Definition}
\theoremstyle{remark}
\newtheorem{remark}[theorem]{Remark}

\newcommand{\E}{\mathbb{E}}
\newcommand{\Prob}{\mathbb{P}}
\newcommand{\R}{\mathbb{R}}

\newcommand{\Var}{\operatorname{Var}}
\newcommand{\Unif}{\mathrm{Unif}}

\title{\textbf{Chart-Generated Probability Geometry\\and Kolmogorov Expectations}}
\author{Manuela-Simona Cojocea\\[0.35em]
\small Faculty of Mathematics and Computer Science, University of Bucharest\\
\small \texttt{simona.cojocea@fmi.unibuc.ro}}
\date{August 24, 2026}

\begin{document}
\maketitle

\begin{abstract}
This paper develops a geometric reinterpretation of probability in which a
cumulative distribution function is not used merely as a passive coordinate
label, but as a generator of metric geometry. An admissible probability chart
$G:I\to(0,1)$ pulls the Euclidean distance of the probability interval back to
value space through
\[
d_G(x,y)=|G(x)-G(y)|.
\]
This creates a decisive distinction from ordinary coordinate invariance: if
$G$ is replaced by another chart $H$ while the observations and their law are
kept fixed, the induced metric-measure structure changes. Although the naked
metric spaces $(I,d_G)$ and $(I,d_H)$ are abstractly isometric, the fixed data
experience different spacing. Probability charts therefore preserve ordinal
structure while altering metric notions such as distance, dispersion,
boundary proximity, and barycentric centrality.

Averaging linearly in the chart coordinate and pulling the result back defines
\[
b_G(X)=G^{-1}\!\bigl(\mathbb E[G(X)]\bigr),
\]
the Fr\'echet barycenter of $X$ under $d_G$. These functionals coincide with
Kolmogorov--Nagumo means, but the chart-generated viewpoint explains their
chart dependence geometrically: changing the chart changes the loss being
minimized, rather than merely rewriting one fixed centre in new coordinates.
A rigidity result shows that, within normalized probability charts, preserving
the barycenter for every law forces the chart itself to remain unchanged.
More generally, the full family of chart-generated Fr\'echet $p$-centres exhibits
a sharp order--metric phase transition. Absolute-distance centrality ($p=1$)
is exactly the median set and is invariant under every admissible chart,
whereas for every $1<p<\infty$ a single universal analytic deformation of any
baseline chart moves the unique $p$-centre continuously through the entire
open interval between the essential extremes of the law. Under the intrinsic
continuous chart $G=F_X$, all finite Fr\'echet $p$-centres collapse to the
median.

The framework yields chart-induced barycenters beyond classical integrability
assumptions. Under the intrinsic chart the coordinates are uniform and the
barycenter is the median; under an external benchmark chart, tail mismatch is
represented by excess boundary occupation. Laws of large numbers, central
limit theorems, and a non-asymptotic concentration inequality are established
for the transformed functionals. Probability-coordinate moments exist for
every law supported on the chart domain and determine the law uniquely, in
contrast with classical moment non-existence and indeterminacy on unbounded
value spaces. The resulting perspective treats chart selection as an explicit
choice of probabilistic ruler and expectation as a geometry-dependent
barycentric operation.
\end{abstract}

\noindent\textbf{Keywords:} chart-generated geometry; probability coordinates; metric-measure geometry; Kolmogorov--Nagumo means; heavy tails; central tendency.

\medskip
\noindent\textbf{Mathematics Subject Classification (2020):} 60B05; 60E05; 60F05.

\bigskip
\section{Introduction}

Expectation is traditionally defined as an average taken in value space.
This viewpoint is so standard that it is rarely questioned, even in settings
where classical assumptions fail. In particular, for heavy-tailed
distributions, the expected value may not exist, or may fail to capture a
meaningful notion of central behaviour.

The starting point of this paper is that a coordinate-induced averaging
functional depends on the chosen probability chart; classical expectation
itself is not changed. Instead of averaging directly in value space, one may
first transport the random variable through a cumulative distribution
function, average on the probability scale, and pull the result back to value
space. This defines a different, chart-indexed functional.

At this point a fundamental geometric distinction becomes decisive. In the
usual differential-geometric interpretation, a change of coordinates is
\emph{passive}: it changes the numerical description of a fixed geometric
object, while distances, geodesics, and barycenters remain the same once the
metric tensor is transformed correctly. Probability charts play a stronger
role here. Each admissible chart $G:I\to(0,1)$ is used both as a coordinate map
and as a generator of geometry by pulling the Euclidean metric of $(0,1)$ back
to $I$:
\[
d_G(x,y)=|G(x)-G(y)|.
\]
Replacing $G$ by another chart $H$ while keeping the points of $I$, the random
variable $X$, and its law fixed does not merely re-express the same metric. It
assigns the Euclidean metric anew in the $H$-coordinate and therefore produces
a different metric $d_H$. In concise terms, chart replacement is a coordinate change followed by the
reimposition of the Euclidean metric.

This distinction identifies the mathematical content of chart selection. A
probability chart is a \emph{probabilistic ruler}: for $x<y$, the distance
$G(y)-G(x)$ is the amount of reference probability mass placed between the two
points by the distribution inducing $G$. Different charts preserve the order
of observations but alter their spacing. Consequently, ranks and other ordinal
statements remain invariant, whereas distances, midpoints, dispersions,
boundary proximity, and barycentric centres may change. The relevant object is
therefore not merely the abstract metric space $(I,d_G)$, which is always
isometric to $(0,1)$, but the fixed-data metric-measure structure
\[
(I,d_G,P_X).
\]
The data remain in place while the geometry through which they are read is
changed.

This leads to quantities of the form
\[
M_g(X)=g^{-1}\bigl(\mathbb E[g(X)]\bigr),
\]
which coincide with the class of generalized means introduced by
Kolmogorov~\cite{Kolmogorov1930}, Nagumo~\cite{Nagumo1930}, de
Finetti~\cite{deFinetti1931}, and Chisini~\cite{Chisini1929}. These
constructions reappear under several names, including Kolmogorov means,
Kolmogorov--Nagumo means, quasi-arithmetic means, $f$-means, generalized means,
and Chisini means. The diversity of terminology reflects historical
development rather than structural differences.

In their original formulation, these means arise as functional solutions to
axiomatic or algebraic characterizations of averaging procedures. The
perspective adopted here is different. Rather than viewing them as abstract
generalizations of the arithmetic mean, we interpret them as Fr\'echet
barycenters for chart-generated probability metrics. What appears algebraically
as a transformed mean is the result of performing an ordinary Euclidean average
on the probability scale and then reading it back through the inverse chart.
The generalized mean is therefore not simply an alternative aggregation rule:
it is the centre selected by a particular probabilistic geometry.

This perspective has also been revisited in a modern statistical context.
De Carvalho~\cite{deCarvalho2016} extends generalized means from finite
collections to distributions, interpreting
\[
b_g(X)=g^{-1}\bigl(\mathbb E[g(X)]\bigr)
\]
as a population-level functional. The present work builds on this observation
but shifts the emphasis from functional extension to geometric structure. The
central question is not only whether the functional can be defined, but what
geometry makes it the natural barycenter.

Recent work has explored geometric interpretations of quasi-arithmetic means
through information geometry and Bregman manifolds; see
\cite{Nielsen2023QAC} and \cite{Amari2016}. There, transformed-coordinate
averages arise from convex-geometric dual coordinates generated by
Legendre-type functions. The present framework instead restricts the charts to
cumulative distribution functions. This gives the geometry a direct
probabilistic interpretation: distance measures reference probability mass,
and chart choice specifies how resolution is distributed across value space.

The viewpoint also aligns with extensions of expectation based on generalized
means. A notable example is R\'enyi's formulation~\cite{renyi1961}, in which
generalized means deform the aggregation of probability weights in the
definition of entropy. Here the averaging operation remains linear on the
chosen probability scale; what changes is the geometry in which averaging is
performed. The distinction is therefore between deforming the aggregation rule
and selecting a different geometry for a fixed linear aggregation.

A CDF-based chart sends a random variable into the bounded interval $(0,1)$,
whose closure is compact, and compresses extreme values towards its boundary.
The term \emph{probability integral transform} is reserved here for the
intrinsic choice $G=F_X$, in which case $G(X)$ is uniform and contains no
marginal tail signature. For an external benchmark chart, boundary occupation
is instead chart-relative: after location and scale have been controlled,
excess occupation near $0$ and $1$ records mismatch between the law of $X$ and
the benchmark geometry. The probability measure of $X$ is not reweighted or
truncated; the geometry assigned to its support is changed.

Within this framework, the probability barycenter is
\[
b_G(X)=G^{-1}\!\bigl(\mathbb E[G(X)]\bigr),
\]
the pullback of a linear average on the probability scale and, equivalently,
the minimizer of expected squared distance under $d_G$. Embedding this
quadratic construction in the Fr\'echet family
\[
\operatorname*{arg\,min}_{c\in I}\mathbb E\bigl[d_G(X,c)^p\bigr],
\qquad 1\le p<\infty,
\]
reveals a sharp structural dichotomy: at $p=1$ the centre is the median set and
is invariant under every strictly increasing probability chart, while every
$p>1$ is maximally chart-sensitive in the sense that varying the chart can move
the unique centre through the full essential interval of the law. Classical
expectation corresponds to the uniform chart on a bounded interval, the
intrinsic continuous chart collapses all finite $p$-centres to the median, and
other choices produce benchmark-relative centres. This makes chart dependence
a substantive modelling feature rather than a coordinate artefact.

The principal contributions of the paper are the following.
\begin{enumerate}
\item We formalize the distinction between passive coordinate changes and
changes of geometry generated by charts. We show how two charts are related
by a positive rescaling of the one-dimensional metric, explain why the associated
naked metric spaces remain abstractly isometric, and identify
$(I,d_G,P_X)$ as the fixed-data object that changes with the chart.

\item We establish the ordinal-invariance/metric-sensitivity principle:
strictly increasing probability charts preserve order while modifying spacing.
For the full chart-generated Fr\'echet family we prove an order--metric phase
transition: the $p=1$ centre is exactly the chart-invariant median set, whereas
for every $1<p<\infty$ a universal exponential chart scale sweeps the unique
$p$-centre through the entire essential interval of any non-degenerate law.
The quadratic case yields an explicit entropic formula and a
Kullback--Leibler sensitivity identity. We then pair this saturation result
with rigidity: within normalized probability charts, universal preservation of
the quadratic barycenter for every law forces the chart to be unchanged.

\item We establish laws of large numbers, central limit theorems, and a
non-asymptotic concentration inequality for probability-coordinate averages
and their pullbacks, under the stated inverse-regularity conditions.

\item We introduce initial and centred probability-coordinate moments and prove
that they determine every law supported on the chart domain, restoring moment
existence and determinacy on the bounded probability scale.

\item We distinguish intrinsic from benchmark geometry, develop
benchmark-relative boundary diagnostics for tail mismatch, and extend the
framework to multivariate copula coordinates.
\end{enumerate}

The framework does not restore a nonexistent classical mean. If
$\mathbb E|X|=\infty$, the classical strong law for the sample mean is not
available, and for some heavy-tailed laws the sample mean does not converge to
a finite deterministic centre. The objects studied here are different,
chart-dependent location functionals. Their stability comes from bounded
coordinate averaging, while their interpretation comes from the selected
probability geometry and the regularity of its inverse chart.

\medskip

The paper is organized as follows. Section~\ref{sec:prob-coords} introduces
probability charts, develops the distinction between passive coordinate
changes and chart-generated metrics, and proves the order--metric phase
transition for Fr\'echet $p$-centres together with its saturation and intrinsic
collapse consequences. Section~\ref{sec:barycenters} consolidates the
quadratic barycentric construction and its relation to Kolmogorov-type
expectations.
Section~\ref{sec:limit-thms} establishes asymptotic and non-asymptotic results.
Section~\ref{sec:Kolmogorov-moments} develops probability-coordinate moments
and moment determinacy. Section~\ref{sec:heavy-tails} studies benchmark-relative
boundary geometry, Section~\ref{sec:multivariate} gives the copula-coordinate
extension, and Section~\ref{sec:discussion} discusses the foundational
implications and further directions.

\section{Probability coordinates and probability geometry}\label{sec:prob-coords}

Classical expectation is an averaging operator in value space endowed
with Euclidean geometry. When probability mass concentrates in the tails, this geometry may become
poorly adapted to the distribution. Probability coordinates replace value
space geometry by a bounded metric geometry, with compact closure, in which a
barycentric structure is always available.

In this section we introduce the basic geometric framework underlying the
paper. The central idea is stronger than using a continuous cumulative
distribution function as a coordinate label: a probability chart maps value
space into probability space and, by pulling back the Euclidean metric, also
determines how the original values are spaced. Thus the chart is simultaneously
a coordinate map and a generator of geometry.

\subsection{Probability coordinate charts}
The next definition formalizes the notion of a probability coordinate chart
and separates the general construction from its canonical specialization.
\begin{definition}[Probability coordinate chart]\label{def:prob-chart}
Let $\widetilde G:\mathbb R\to[0,1]$ be a continuous cumulative distribution function. Suppose there exists an interval
\(
I=(a,b)\subseteq\mathbb R,
\)
such that: 

\begin{enumerate}
\item the restriction
\(
G=\widetilde G|_I
\)
is strictly increasing;

\item
\(
\lim_{x\to a^+}G(x)=0,
\qquad
\lim_{x\to b^-}G(x)=1.
\)
\end{enumerate}

Then the restriction
\(
G:I\rightarrow(0,1)
\)
is called the probability coordinate chart induced by $\widetilde G$.
\end{definition}

\begin{definition}[Probability barycenter]\label{def:prob-barycenter}
Given an $I$-valued random variable $X$, we define the
\emph{probability barycenter} by
\[
b_G(X):=G^{-1}\!\bigl(\mathbb{E}(G(X))\bigr).
\]
It is also called the \emph{Kolmogorov expectation induced by $G$}.
\end{definition}

\begin{remark}[Uniform chart]
The requirement that a probability-coordinate chart be strictly increasing
should be understood relative to the interval on which the chart is used,
rather than necessarily on the whole real line.

For example, the cumulative distribution function of the uniform law on
$[a,b]$ is
\[
G(x)=
\begin{cases}
0, & x\le a,\\[1mm]
\dfrac{x-a}{b-a}, & a<x<b,\\[2mm]
1, & x\ge b.
\end{cases}
\]
Viewed as a function on $\mathbb{R}$, $G$ is not strictly increasing because
it is constant outside $[a,b]$. Following the convention adopted in this
paper, we restrict attention to the maximal open interval on which $G$ is
strictly increasing,
\[
I=(a,b).
\]
On this interval,
\[
G_I(x)=\frac{x-a}{b-a}
\]
defines a bijection from $(a,b)$ onto $(0,1)$ and therefore qualifies as an
admissible chart.

If $X$ is a continuous random variable supported on $[a,b]$, then
\[
G_I(X)=\frac{X-a}{b-a},
\]
so the probability-coordinate representation reduces to a linear rescaling
of value space.

The associated probability barycenter is
\[
b_{G_I}(X)
=
G_I^{-1}\!\left(\mathbb{E}\!\left[G_I(X)\right]\right).
\]
Since
\[
G_I^{-1}(u)=a+(b-a)u,
\]
we obtain
\[
b_{G_I}(X)
=
a+(b-a)\,
\mathbb{E}\!\left(\frac{X-a}{b-a}\right)
=
\mathbb{E}(X).
\]

Thus the arithmetic mean appears as a particular probability barycenter,
namely the barycenter associated with the uniform chart. In this sense,
classical expectation is not external to the probability-coordinate
framework but is recovered as a distinguished choice of probability
geometry.

From a geometric perspective, the uniform chart may be viewed as a neutral
coordinate system. Its derivative
\[
G_I'(x)=\frac{1}{b-a}
\]
is constant, meaning that all regions of the interval receive the same
coordinate resolution. Unlike Gaussian, logistic, Student, or Cauchy
charts, the uniform chart does not preferentially emphasize central values,
tails, or any other specific geometric features.
\end{remark}

\begin{remark}[The inverse chart as geometric backbone]
The probability chart \(G\) sends values into probability coordinates, while
the inverse chart \(G^{-1}\) brings probability coordinates back to value
space. Thus, although the averaging operation is performed on the bounded
scale \(G(X)\in(0,1)\), its final interpretation is entirely determined by
the pullback map
\[
        G^{-1}:(0,1)\longrightarrow \mathbb R .
\]
In this sense, \(G^{-1}\) is not merely a technical inverse. It describes how
the canonical probability scale is unfolded into value space.

This observation is already visible in the density transformation
\[
        f_{G(X)}(u)
        =
        \frac{f_X(G^{-1}(u))}
             {g(G^{-1}(u))},
        \qquad 0<u<1,
\]
whenever $X$ has density $f_X$, $G$ is differentiable with density
$g=G'$, and $g(G^{-1}(u))>0$. The shape of the
probability-coordinate representation \(G(X)\) is therefore governed by the
interaction between the original density \(f_X\) and the geometry induced by
\(G^{-1}\). Regions where \(G^{-1}\) is steep correspond to an expansion of
probability coordinates in value space, while flatter regions correspond to
compression.

Consequently, the probability barycenter
\[
        b_G(X)=G^{-1}\bigl(\mathbb E(G(X))\bigr)
\]
should be understood as the result of two distinct operations: first a
barycentric averaging on the probability scale, and then a geometric
interpretation through the inverse chart. This separation clarifies why
different choices of \(G\) may produce different notions of centrality: they
do not merely transform the data, but generate different metrics on the fixed
value space, as formalized in Section~\ref{sec:chart-generated-geometry}. The
inverse chart then determines how a coordinate barycenter is interpreted in
that chart-generated geometry.
\end{remark}

\begin{figure}[H]
\centering
\resizebox{\textwidth}{!}{%
\begin{tikzpicture}[
    >=Latex,
    every node/.style={font=\small},
    point/.style={circle, fill=black, inner sep=1.5pt},
    obs/.style={circle, fill=blue!70!black, inner sep=1pt},
    probobs/.style={circle, fill=purple!75!black, inner sep=1pt},
    map/.style={->, thick},
    axis/.style={->, thick}
]

\node[font=\bfseries\large] at (0,4.2) {Value space};
\node at (0,3.8) {$\mathbb{R}$};

\draw[axis] (-3.6,0) -- (3.6,0) node[right] {$x$};

\draw[blue!70!black, thick, smooth, domain=-3.2:3.2, samples=120]
plot (\x,{2.2/(1+3*\x*\x)});

\node[blue!70!black] at (-2.3,1.1) {$f_X(x)$};

\foreach \x in {-3,-2.7,-2.5,-2.2,-2,-1.7,-1.3,-0.6,-0.4,-0.2,0.2,0.4,0.6,1.3,1.7,2.1,2.5,2.8,3.1}
    \node[obs] at (\x,0.06) {};

\node[point] (xaxis) at (0.6,0) {};
\node[point] (xdens) at (0.6,{2.2/(1+3*0.6*0.6)}) {};
\draw[dashed] (xaxis) -- (xdens);
\node[below] at (0.6,-0.1) {$x$};

\node[align=center] at (0,-1.1)
{Observations $X$ in value space.\\
Heavy tails extend to $\pm\infty$.};

\node[font=\bfseries\large] at (8,4.2) {Probability coordinates};
\node at (8,3.8) {$(0,1)$};

\draw[thick] (5,0) -- (11,0);
\node[below] at (5,-0.1) {$0$};
\node[below] at (11,-0.1) {$1$};

\foreach \u in {5.05,5.1,5.2,5.3,5.4,5.6,5.8}
    \node[probobs] at (\u,{0.05+0.02*rnd}) {};
\foreach \u in {10.2,10.4,10.5,10.6,10.7,10.8,10.9}
    \node[probobs] at (\u,{0.05+0.02*rnd}) {};

\node[point] (u) at (8.5,0) {};
\draw[dashed] (8.5,0) -- (8.5,-0.6);
\node[below] at (9.2,-0.15) {$u=G(x)$};

\node[point] (ubar) at (8,0) {};
\draw[->, thick, purple!80!black] (8,1.0) -- (8,0.1);
\node[purple!80!black] at (8,1.2) {$\bar u=\mathbb{E}(G(X))$};

\node[align=center] at (8.4,2.3)
{Coordinate variable $U=G(X)$.\\
Tail mismatch relative to $G$ $\rightarrow$ boundary occupation.};

\node[align=center] at (8,-1.1)
{Averaging in probability coordinates\\
(always well-defined).};

\draw[map] (2.5,1.4) .. controls (4.2,2.2) and (5.8,2.2) ..
node[above] {$G$} (7,1.4);

\draw[map] (7,-1.4) .. controls (5.8,-2.1) and (4.2,-2.1) ..
node[below] {$G^{-1}$} (2.5,-1.4);

\node[probobs, inner sep=2pt] (eg) at (0,-2.4) {};
\draw[dashed] (eg) -- (0,-0.2);

\node at (0,-2.7) {$b_G(X)=G^{-1}(\bar u)$};

\node[align=center] at (4,-3.6)
{Pull-back through $G^{-1}$ gives the\\
probability barycenter in value space.};

\end{tikzpicture}
}

\caption{Construction of the probability barycenter. Observations \(X\) in value space are mapped by a probability coordinate chart \(G\) into the unit interval \((0,1)\), where averaging is performed. The resulting coordinate barycenter \(\bar u=\mathbb{E}(G(X))\) is then pulled back through \(G^{-1}\) to value space, yielding the probability barycenter \(b_G(X)\). For a benchmark chart, occupation near \(0\) and \(1\) describes tail behaviour relative to the benchmark; for the continuous intrinsic chart, the coordinates are uniform.}
\label{fig:probability-coordinate-construction}
\end{figure}

\subsection{Chart-generated metrics: coordinate changes versus changes of geometry}
\label{sec:chart-generated-geometry}

The distinction between a passive re-expression of one geometry and the
selection of a new chart-generated geometry is the structural centre of the
framework.

\begin{definition}[Chart-generated distance]\label{def:induced-distance}
Let $G:I\to(0,1)$ be a probability coordinate chart. The distance generated by
$G$ is
\[
d_G(x,y):=|G(x)-G(y)|,\qquad x,y\in I.
\]
If $G$ is continuously differentiable with $G'(x)>0$ on $I$, the corresponding
one-dimensional Riemannian line element is
\[
ds_G^2=\bigl(G'(x)\bigr)^2\,dx^2.
\]
\end{definition}

\begin{remark}[A probability chart as a probabilistic ruler]
Let $Z_G$ have the continuous distribution function inducing the chart $G$.
For $x<y$ in $I$,
\[
d_G(x,y)=G(y)-G(x)=\Prob(x<Z_G\le y).
\]
Thus geometric length is reference probability mass. Intervals of equal
$G$-length contain equal mass under the chart law, even when their Euclidean
widths are very different. Choosing $G$ therefore specifies which portions of
value space are expanded, compressed, or placed close to the probability
boundary.
\end{remark}

\begin{proposition}[Passive reparametrization versus chart replacement]
\label{prop:passive-active}
Let $G,H:I\to(0,1)$ be probability coordinate charts on the same interval and
set
\[
T:=H\circ G^{-1}:(0,1)\to(0,1).
\]
Then $H=T\circ G$ and
\[
d_H(x,y)=\bigl|T(G(x))-T(G(y))\bigr|.
\]
If $G$ and $H$ are continuously differentiable with strictly positive
derivatives, then
\[
ds_H^2=\bigl(T'(G(x))\bigr)^2ds_G^2.
\]
By contrast, if $v=H(x)$ is used merely as a new coordinate in which to express
the already fixed geometry $d_G$, then
\[
ds_G^2
=
\left((G\circ H^{-1})'(v)\right)^2dv^2,
\]
not $dv^2$ in general.
\end{proposition}

\begin{proof}
The identity $H=T\circ G$ follows from the definition of $T$, and the distance
formula follows immediately. Differentiating gives
$H'(x)=T'(G(x))G'(x)$, which yields the relation between the two line elements.
For the passive coordinate expression, write
$G(x)=(G\circ H^{-1})(v)$ and differentiate.
\end{proof}

\begin{remark}[The metric-reassignment step]
In an ordinary coordinate change, the coefficient
$((G\circ H^{-1})'(v))^2$ is retained, so the underlying geometry is unchanged.
The probability-coordinate framework performs a different operation when it
replaces $G$ by $H$: after passing to $v=H(x)$, it equips the $v$-scale anew with
the standard Euclidean metric $dv^2$. Pulling that metric back produces $d_H$.
Hence chart replacement is not merely a change of notation; it is a coordinate
change followed by the reimposition of the Euclidean metric.
\end{remark}

\begin{figure}[H]
\centering
\resizebox{\textwidth}{!}{%
\begin{tikzpicture}[
    >=Latex,
    every node/.style={font=\small},
    box/.style={draw, rounded corners, align=center, minimum height=1.15cm,
                text width=4.25cm, inner sep=6pt},
    arrow/.style={->, thick}
]

\node[font=\bfseries] at (-6.1,2.05) {Passive coordinate change};
\node[box, draw=purple!80!black] (pg1) at (-6.1,0.8)
{$u=G(x)$\\ $ds_G^2=du^2$};
\node[box, draw=purple!80!black] (pg2) at (0,0.8)
{$v=H(x)=T(u)$\\
$ds_G^2=((T^{-1})'(v))^2dv^2$};
\draw[arrow, purple!80!black] (pg1) -- node[above]{same geometry} (pg2);

\node[font=\bfseries] at (-6.1,-1.25) {Chart replacement};
\node[box, draw=orange!85!black] (cg1) at (-6.1,-2.5)
{fixed observations $X$\\ fixed law $P_X$};
\node[box, draw=orange!85!black] (cg2) at (0,-2.5)
{choose $H$ and declare\\ $ds_H^2=dv^2$};
\node[box, draw=orange!85!black] (cg3) at (6.1,-2.5)
{new fixed-data geometry\\ $(I,d_H,P_X)$};
\draw[arrow, orange!85!black] (cg1) -- node[above]{metric choice} (cg2);
\draw[arrow, orange!85!black] (cg2) -- node[above]{pullback} (cg3);

\end{tikzpicture}
}
\caption{Passive reparametrization versus chart-generated geometry. In the top
row, the metric coefficient is transformed together with the coordinate, so
the geometry remains $d_G$. In the bottom row, the observations and their law
are kept fixed while the new chart coordinate is equipped with a fresh
Euclidean metric; pulling it back produces the genuinely different geometry
$d_H$.}
\label{fig:passive-active-chart-change}
\end{figure}

\begin{proposition}[Abstract isometry and fixed-data geometry]
\label{prop:marked-isometry}
For any two probability coordinate charts $G$ and $H$ on $I$, the map
\[
\Phi_{G,H}:=H^{-1}\circ G
\]
is an isometry from $(I,d_G)$ onto $(I,d_H)$. More precisely,
\[
d_H\bigl(\Phi_{G,H}(x),\Phi_{G,H}(y)\bigr)=d_G(x,y).
\]
Consequently,
\[
(I,d_G,P_X)
\quad\text{is isomorphic to}\quad
(I,d_H,(\Phi_{G,H})_{\#}P_X),
\]
but not in general to $(I,d_H,P_X)$.
\end{proposition}

\begin{proof}
Since $H(\Phi_{G,H}(x))=G(x)$,
\[
d_H\bigl(\Phi_{G,H}(x),\Phi_{G,H}(y)\bigr)
=
|G(x)-G(y)|
=
d_G(x,y).
\]
An isometry transports the accompanying measure by pushforward. In general,
\[
(\Phi_{G,H})_{\#}P_X\neq P_X,
\]
so keeping $P_X$ fixed produces a different metric-measure structure.
\end{proof}

\begin{remark}[Why abstract isometry does not remove chart dependence]
Every $(I,d_G)$ is isometric to the ordinary interval $(0,1)$, so the naked
spaces have the same abstract metric type. The framework, however, does not
transport the observations through $\Phi_{G,H}$ when a benchmark chart is
changed. It compares the same labelled values and the same law under different
metrics. The relevant object is therefore the marked metric-measure structure
$(I,d_G,P_X)$, not the unlabelled isometry class of $(I,d_G)$.
\end{remark}

\begin{corollary}[Ordinal invariance and metric sensitivity]
\label{cor:ordinal-metric}
All probability charts preserve the order on $I$:
\[
x<y\quad\Longleftrightarrow\quad G(x)<G(y).
\]
Hence rank statements and purely ordinal summaries are chart-invariant. In
contrast, metric constructions such as $d_G(x,y)$, chart midpoints
\[
m_G(x,y)=G^{-1}\!\left(\frac{G(x)+G(y)}{2}\right),
\]
Fr\'echet barycenters and variances, as well as boundary distances, are
generally chart-dependent.
\end{corollary}

\begin{proposition}[Rigidity of the chart-generated metric]
\label{prop:metric-rigidity}
Let $G,H:I\to(0,1)$ be increasing probability coordinate charts. If
\[
d_G(x,y)=d_H(x,y)
\qquad\text{for every }x,y\in I,
\]
then $G=H$.
\end{proposition}

\begin{proof}
For $x<y$, equality of distances gives
$G(y)-G(x)=H(y)-H(x)$. Thus $H-G$ is constant on $I$. Since both charts tend
to $0$ at the left endpoint and to $1$ at the right endpoint, the constant is
zero.
\end{proof}

\begin{proposition}[Geometric characterization]\label{prop:geometric-characterization}
Assume $\mathbb{E}(G(X))$ exists. Then
\[
b_G(X) = \arg\min_{c \in I} \mathbb{E}\big( d_G(X,c)^2 \big).
\]
\end{proposition}

\begin{proof}
We write
\[
\mathbb{E}(d_G(X,c)^2) = \mathbb{E}((G(X) - G(c))^2).
\]
This is minimized when $G(c) = \mathbb{E}(G(X))$, hence
\[
c = G^{-1}(\mathbb{E}(G(X))).
\]
\end{proof}

\subsection{Fr\'echet \texorpdfstring{$p$}{p}-centres and the order--metric phase transition}
\label{subsec:order-metric-phase-transition}

The quadratic barycenter is one member of a larger family of centres generated
by the same probability geometry. Extending from squared distance to arbitrary
powers makes it possible to separate what is forced by order from what is
created by the metric.

\begin{definition}[Chart-generated Fr\'echet $p$-centre]
\label{def:frechet-p-centre}
Let $1\le p<\infty$, let $G:I\to(0,1)$ be an admissible probability chart,
and let $X$ be $I$-valued. Define
\[
\mathcal C_{G,p}(X)
:=
\operatorname*{arg\,min}_{c\in I}
\mathbb E\bigl[d_G(X,c)^p\bigr].
\]
For $p>1$ the minimizer is unique; it will be denoted by $c_{G,p}(X)$.
For $p=1$ the minimizer may be set-valued.
\end{definition}

\begin{definition}[Essential infimum, essential supremum, and essential interval]
\label{def:essential-endpoints}
Let $X$ be a real-valued random variable on $(\Omega,\mathcal F,\Prob)$. Its
\emph{essential infimum} and \emph{essential supremum} are
\begin{align}
\operatorname*{ess\,inf}X
&:=\sup\bigl\{a\in\mathbb R:\Prob(X<a)=0\bigr\}
 =\sup\bigl\{a\in\mathbb R:X\ge a\ \text{a.s.}\bigr\},
\label{eq:essinf-definition}\\
\operatorname*{ess\,sup}X
&:=\inf\bigl\{b\in\mathbb R:\Prob(X>b)=0\bigr\}
 =\inf\bigl\{b\in\mathbb R:X\le b\ \text{a.s.}\bigr\},
\label{eq:esssup-definition}
\end{align}
with values in the extended real line allowed. We write
\[
\underline x:=\operatorname*{ess\,inf}X,
\qquad
\overline x:=\operatorname*{ess\,sup}X,
\]
and call $(\underline x,\overline x)$ the \emph{essential interval} of a
non-degenerate law, equivalently the interior of its essential convex hull.
The qualifier ``essential'' means that modifications on null events are
ignored: if $X=Y$ almost surely, then $X$ and $Y$ have the same essential
endpoints. The endpoints need not be attained. When finite,
\[
\Prob(X<\underline x)=0,
\qquad
\Prob(X<\underline x+\varepsilon)>0\quad(\varepsilon>0),
\]
and similarly
\[
\Prob(X>\overline x)=0,
\qquad
\Prob(X>\overline x-\varepsilon)>0\quad(\varepsilon>0).
\]
These are the standard measure-theoretic essential bounds; see, for example,
Folland~\cite[Sec.~6.1, p.~184]{Folland1999} for the essential-supremum
construction underlying $L^\infty$.
\end{definition}

\begin{remark}[Essential interval versus support]
\label{rem:essential-interval-support}
The essential interval should not be confused with the topological support.
For example, if $\Prob(X=0)=\Prob(X=10)=1/2$, then
$\operatorname{supp}(X)=\{0,10\}$ but
$(\operatorname*{ess\,inf}X,\operatorname*{ess\,sup}X)=(0,10)$.
Thus a chart-generated Fr\'echet centre may move through gaps in the support;
this is why the saturation theorem below is stated in terms of the essential
interval rather than the support.
\end{remark}

If $U=G(X)$, then boundedness of $U$ makes the objective finite for every
$p\ge1$, without any moment assumption on $X$. Since $G$ is a strictly
increasing bijection,
\[
\mathcal C_{G,p}(X)
=
G^{-1}\!\left(
\operatorname*{arg\,min}_{z\in(0,1)}\mathbb E|U-z|^p
\right).
\]
For $p>1$, write
\[
\nu_p(U):=
\operatorname*{arg\,min}_{z\in\mathbb R}\mathbb E|U-z|^p.
\]
Strict convexity gives uniqueness, and differentiation of the coordinate loss
gives the balance equation
\begin{equation}\label{eq:p-balance}
\mathbb E\bigl[(U-\nu_p(U))_+^{p-1}\bigr]
=
\mathbb E\bigl[(\nu_p(U)-U)_+^{p-1}\bigr].
\end{equation}

\begin{lemma}[Convex coordinate comparison]
\label{lem:convex-coordinate-comparison}
Let $Y$ be a non-degenerate bounded random variable, let $1<p<\infty$, and
let $\phi$ be strictly increasing and convex on an interval containing the
essential range of $Y$. Then
\[
\nu_p(\phi(Y))\ge \phi(\nu_p(Y)).
\]
If $\phi$ is strictly convex on that interval, then
\[
\nu_p(\phi(Y))>\phi(\nu_p(Y)).
\]
For strictly increasing concave $\phi$, the inequalities reverse.
\end{lemma}

\begin{proof}
Put $m=\nu_p(Y)$ and $r=p-1$. The balance equation gives
\[
\mathbb E[(m-Y)_+^r]=\mathbb E[(Y-m)_+^r].
\]
Choose a subgradient $s\in\partial\phi(m)$. Convexity yields
\[
\phi(m)-\phi(y)\le s(m-y),\qquad y<m,
\]
and
\[
\phi(y)-\phi(m)\ge s(y-m),\qquad y>m.
\]
Hence
\[
\mathbb E[(\phi(m)-\phi(Y))_+^r]
\le
s^r\mathbb E[(m-Y)_+^r]
=
s^r\mathbb E[(Y-m)_+^r]
\le
\mathbb E[(\phi(Y)-\phi(m))_+^r].
\]
Thus the derivative of
$z\mapsto\mathbb E|\phi(Y)-z|^p$ at $z=\phi(m)$ is non-positive, so its
unique minimizer cannot lie to the left of $\phi(m)$. If $\phi$ is strictly
convex, the supporting-line inequalities are strict away from $m$; because a
non-degenerate $p$-centre lies strictly between the essential extremes, there
is positive mass on both sides of $m$, and the final comparison is strict.
The concave case follows by reversing the inequalities.
\end{proof}

For $t\in\mathbb R$, define the probability-scale deformation
\begin{equation}\label{eq:exponential-chart-deformation}
T_t(u)=
\begin{cases}
\displaystyle\frac{e^{tu}-1}{e^t-1},&t\ne0,\\[7pt]
u,&t=0,
\end{cases}
\qquad 0\le u\le1.
\end{equation}
Each $T_t$ is an increasing homeomorphism of $[0,1]$ fixing the endpoints.
If $H$ is induced by the continuous CDF $\widetilde H$, then
$T_t\circ\widetilde H$ is again a continuous CDF. Consequently, if $H$ is
an admissible probability chart, then
\[
G_t:=T_t\circ H
\]
is again an admissible probability chart. The family is ordered by convexity:
if $s>t$, then $T_s\circ T_t^{-1}$ is strictly convex. Indeed, writing
\[
a_t:=
\begin{cases}
\dfrac{t}{e^t-1},&t\ne0,\\[4pt]
1,&t=0,
\end{cases}
\]
we have $T_t'(u)=a_t e^{tu}$ and therefore
\[
\frac{d}{dy}\bigl(T_s\circ T_t^{-1}\bigr)(y)
=
\frac{a_s}{a_t}e^{(s-t)T_t^{-1}(y)},
\]
which is strictly increasing in $y$.

\begin{theorem}[Order--metric phase transition and barycentric saturation]
\label{thm:order-metric-phase-transition}
Let $X$ be a non-degenerate $I$-valued random variable and let
$\underline x=\operatorname*{ess\,inf}X$ and
$\overline x=\operatorname*{ess\,sup}X$ be its essential endpoints as in
Definition~\ref{def:essential-endpoints}, with endpoint limits understood in
the extended real sense. Then:

\begin{enumerate}[label=\textup{(\roman*)}]
\item For $p=1$ and every admissible probability chart $G$,
\begin{equation}\label{eq:p1-median-invariance}
\mathcal C_{G,1}(X)=\operatorname{Med}(X),
\end{equation}
where
\[
\operatorname{Med}(X)
=
\left\{m\in I:
\Prob(X\le m)\ge\frac12,
\ \Prob(X\ge m)\ge\frac12
\right\}.
\]
Thus absolute-distance centrality is exactly chart-invariant.

\item Fix any $1<p<\infty$ and any baseline admissible chart $H$. With
$G_t=T_t\circ H$ defined by \eqref{eq:exponential-chart-deformation}, the map
\[
\mathbb R\ni t\longmapsto c_{G_t,p}(X)
\]
is a continuous strictly increasing bijection from $\mathbb R$ onto
\begin{equation}\label{eq:p-orbit-interval}
(\underline x,\overline x).
\end{equation}
Consequently,
\begin{equation}\label{eq:complete-p-orbit}
\left\{c_{G,p}(X):G\text{ admissible}\right\}
=
(\underline x,\overline x),
\qquad 1<p<\infty.
\end{equation}
\end{enumerate}
\end{theorem}

\begin{proof}
For $p=1$, the coordinate objective is $z\mapsto\mathbb E|G(X)-z|$,
whose minimizer set is the median set of $G(X)$. A strictly increasing
transformation preserves the two median inequalities, so
\[
\operatorname{Med}(G(X))=G(\operatorname{Med}(X)).
\]
Pulling back by $G^{-1}$ proves \eqref{eq:p1-median-invariance}.

Now let $p>1$, put $U=H(X)$, and write
\[
u_-:=\operatorname*{ess\,inf}U,
\qquad
u_+:=\operatorname*{ess\,sup}U.
\]
Define
\[
q_p(t)
:=
T_t^{-1}\!\left(\nu_p(T_t(U))\right),
\]
so that
\begin{equation}\label{eq:cp-via-q}
c_{G_t,p}(X)=H^{-1}(q_p(t)).
\end{equation}
If $s>t$, then $\Phi_{s,t}:=T_s\circ T_t^{-1}$ is strictly increasing and
strictly convex. Applying Lemma~\ref{lem:convex-coordinate-comparison} to
$Y=T_t(U)$ gives
\[
\nu_p(T_s(U))
=
\nu_p(\Phi_{s,t}(T_t(U)))
>
\Phi_{s,t}(\nu_p(T_t(U))).
\]
Applying $T_s^{-1}$ yields $q_p(s)>q_p(t)$. Thus $q_p$, and hence
$t\mapsto c_{G_t,p}(X)$, is strictly increasing.

Continuity follows from boundedness. If $t_n\to t$, then
$T_{t_n}\to T_t$ uniformly on $[0,1]$. Therefore the strictly convex losses
\[
z\longmapsto\mathbb E|T_{t_n}(U)-z|^p
\]
converge uniformly on $[0,1]$ to the corresponding loss at $t$. Their unique
minimizers converge, and continuity of $(t,u)\mapsto T_t^{-1}(u)$ gives
$q_p(t_n)\to q_p(t)$.

It remains to identify the endpoint limits. Since $q_p(t)\in(u_-,u_+)$ and
$q_p$ is increasing, the limit $L_+=\lim_{t\to\infty}q_p(t)$ exists with
$L_+\le u_+$. Suppose that $L_+<u_+$. Choose $\delta>0$ such that
$L_++2\delta<u_+$ and
\[
\Prob(U>L_++2\delta)>0.
\]
For sufficiently large positive $t$, $q_p(t)\le L_++\delta$. With
$r=p-1$ and $q=q_p(t)$, the balance equation for $T_t(U)$, after cancelling
the common factor $(e^t-1)^{-r}$, becomes
\[
\mathbb E[(e^{tU}-e^{tq})_+^r]
=
\mathbb E[(e^{tq}-e^{tU})_+^r].
\]
The right-hand side is at most $e^{rtq}$. On the event
$A=\{U>L_++2\delta\}$ we have $U-q\ge\delta$, so the left-hand side is at
least
\[
\Prob(A)e^{rtq}(e^{t\delta}-1)^r,
\]
which eventually exceeds $e^{rtq}$, a contradiction. Hence
\[
\lim_{t\to\infty}q_p(t)=u_+.
\]
Using the identity $T_{-t}(u)=1-T_t(1-u)$ and applying the same argument to
$1-U$ gives
\[
\lim_{t\to-\infty}q_p(t)=u_-.
\]
Thus $q_p$ is a continuous strictly increasing bijection
$\mathbb R\to(u_-,u_+)$. Since $H$ is increasing and continuous,
\[
\lim_{u\downarrow u_-}H^{-1}(u)=\underline x,
\qquad
\lim_{u\uparrow u_+}H^{-1}(u)=\overline x,
\]
with one-sided limits taken in the extended real sense when necessary.
Equation~\eqref{eq:cp-via-q} therefore proves the claimed bijection onto
$(\underline x,\overline x)$.

Finally, every $p>1$ centre under any admissible chart lies strictly between
$\underline x$ and $\overline x$, because the unique coordinate $p$-centre
lies strictly between the essential extremes of the non-degenerate bounded
variable $G(X)$. The exponential family already attains every point in this
interval, which proves \eqref{eq:complete-p-orbit}.
\end{proof}

\begin{remark}[Essential interval, not support]
The orbit in \eqref{eq:complete-p-orbit} is the open interval between the
essential extremes, not necessarily the support itself. For example, if
$\Prob(X=0)=\Prob(X=10)=1/2$, the support is $\{0,10\}$, but for every
$p>1$ the chart orbit of the unique Fr\'echet $p$-centre is the whole interval
$(0,10)$. The theorem therefore describes maximal freedom inside the essential
convex hull of the law.
\end{remark}

\begin{corollary}[Quadratic saturation and information sensitivity]
\label{cor:quadratic-saturation-KL}
Under the assumptions of Theorem~\ref{thm:order-metric-phase-transition}, let
$p=2$, fix a baseline chart $H$, put $U=H(X)$, and write $b_t=b_{G_t}(X)$.
Then, for $t\ne0$,
\begin{equation}\label{eq:entropic-barycenter-path}
H(b_t)
=
\frac1t\log\mathbb E(e^{tU}),
\end{equation}
with $H(b_0)=\mathbb E(U)$. If
\[
K(t)=\log\mathbb E(e^{tU})
\]
and the exponentially tilted law $P_t$ is defined by
\[
\frac{dP_t}{dP}=e^{tU-K(t)},
\]
then
\begin{equation}\label{eq:KL-barycenter-speed}
\frac{d}{dt}H(b_t)
=
\frac{D_{\mathrm{KL}}(P_t\Vert P)}{t^2}>0,
\qquad t\ne0,
\end{equation}
and
\begin{equation}\label{eq:baseline-coordinate-speed}
\left.\frac{d}{dt}H(b_t)\right|_{t=0}
=
\frac12\Var(U).
\end{equation}
If, in addition, $H$ is differentiable and $H'(b_t)>0$, then
\begin{equation}\label{eq:value-space-barycenter-speed}
b_t'
=
\frac{D_{\mathrm{KL}}(P_t\Vert P)}{t^2H'(b_t)},
\qquad t\ne0,
\end{equation}
while
\[
b_0'
=
\frac{\Var(H(X))}{2H'(b_H(X))}.
\]
\end{corollary}

\begin{proof}
For $t\ne0$,
\[
\mathbb E[T_t(U)]
=
\frac{\mathbb E(e^{tU})-1}{e^t-1},
\qquad
T_t^{-1}(v)
=
\frac1t\log\bigl(1+(e^t-1)v\bigr).
\]
Therefore
\[
T_t^{-1}(\mathbb E[T_t(U)])
=
\frac1t\log\mathbb E(e^{tU}),
\]
which gives \eqref{eq:entropic-barycenter-path}. If
$m(t)=K(t)/t$, then
\[
m'(t)=\frac{tK'(t)-K(t)}{t^2}.
\]
But
\[
D_{\mathrm{KL}}(P_t\Vert P)
=
\mathbb E_{P_t}\!\left[\log\frac{dP_t}{dP}\right]
=tK'(t)-K(t),
\]
which proves \eqref{eq:KL-barycenter-speed}. Since $U$ is bounded,
$K$ is analytic at zero and
\[
K(t)=t\mathbb E(U)+\frac{t^2}{2}\Var(U)+O(t^3),
\]
giving \eqref{eq:baseline-coordinate-speed}. The value-space identities follow
by differentiating $H(b_t)=m(t)$.
\end{proof}

\begin{corollary}[Intrinsic Fr\'echet collapse]
\label{cor:intrinsic-frechet-collapse}
Suppose that $X$ has a continuous strictly increasing distribution function
$F_X:I\to(0,1)$. Then, for every $1\le p<\infty$,
\begin{equation}\label{eq:intrinsic-all-p-collapse}
c_{F_X,p}(X)
=
F_X^{-1}\!\left(\frac12\right),
\end{equation}
where for $p=1$ the centre is unique. Thus all finite Fr\'echet exponents
collapse to the median under the intrinsic probability geometry.
\end{corollary}

\begin{proof}
The intrinsic coordinate $U=F_X(X)$ is uniform on $(0,1)$. Its distribution is
symmetric about $1/2$, so the minimizer of
$z\mapsto\mathbb E|U-z|^p$ is $1/2$ for every $p>1$, while its unique median is
also $1/2$. Pullback through $F_X^{-1}$ yields
\eqref{eq:intrinsic-all-p-collapse}.
\end{proof}

\begin{remark}[The Chebyshev endpoint]
The same geometric side of the phase transition persists at $p=\infty$ if one
defines the centre by minimizing
\[
\operatorname*{ess\,sup}_{\omega}\, d_G(X(\omega),c).
\]
In probability coordinates the minimax centre is the midpoint of the essential
extremes. Under the exponential chart scale $G_t=T_t\circ H$, its pullback
moves continuously and strictly from $\underline x$ to $\overline x$.
Thus $p=1$ is the exceptional chart-rigid Fr\'echet exponent, whereas the
range $1<p\le\infty$ lies on the chart-sensitive side.
\end{remark}

\begin{remark}[Relation to classical mean theory]
For $p=2$, the coordinate path in
\eqref{eq:entropic-barycenter-path} belongs algebraically to the classical
theory of quasi-arithmetic scales; see, for example,
\cite{Pasteczka2015}. Comparison theory for broader quasideviation means has a
long history; see \cite{Pales1988}. The role of
Theorem~\ref{thm:order-metric-phase-transition} in the present framework is the
probability-geometric classification across Fr\'echet exponents: absolute loss
is completely chart-invariant, while every strictly convex power loss admits
maximal displacement under admissible probability geometries.
\end{remark}

\begin{proposition}[Passive covariance of probability barycenters]
\label{prop:passive-covariance}
Let $G$ and $H$ be probability coordinate charts on $I$, and let
$\Phi_{G,H}=H^{-1}\circ G$. Then
\[
b_H\bigl(\Phi_{G,H}(X)\bigr)
=
\Phi_{G,H}\bigl(b_G(X)\bigr).
\]
\end{proposition}

\begin{proof}
Since $H(\Phi_{G,H}(X))=G(X)$,
\[
b_H\bigl(\Phi_{G,H}(X)\bigr)
=H^{-1}\!\bigl(\mathbb E[G(X)]\bigr)
=H^{-1}\!\circ G\bigl(b_G(X)\bigr).
\]
\end{proof}

\begin{remark}[What coordinate invariance would mean]
Proposition~\ref{prop:passive-covariance} is the genuine covariance statement:
the chart, the data, and the centre are transported together. The comparison
of substantive interest in this paper is instead
\[
b_G(X)\quad\text{versus}\quad b_H(X),
\]
where $X$ is held fixed and the metric changes. Writing $H=T\circ G$, the
second centre can be expressed in $G$-coordinates as
\[
b_H(X)
=
G^{-1}\!\left(
T^{-1}\!\left(\mathbb E[T(G(X))]\right)
\right).
\]
Thus changing charts replaces the arithmetic mean of $G(X)$ by the
quasi-arithmetic mean generated by $T$; this is a change of barycentric
geometry, not a coordinate artefact.
\end{remark}

\begin{remark}
In this framework, the statistical functional arises as a consequence of the
geometry generated by the probability chart. The chart does not merely supply
coordinates in which a pre-existing centre is rewritten; it determines the
squared-distance loss whose minimizer defines the centre.
\end{remark}

\begin{remark}
When $G = F_X$, the cumulative distribution function of a random variable $X$, 
this construction recovers the canonical probability coordinate system induced by $X$.
\end{remark}
\begin{remark}[Scope of the framework]
The probability coordinate chart $G$ is assumed to be a continuous cumulative distribution 
function that is strictly increasing on an interval $I$, ensuring the existence of a classical inverse.

Subject only to $\Prob(X\in I)=1$, the random variable $X$ may be continuous,
discrete, or mixed. The transformation $G(X)$ is then well-defined and bounded,
and therefore integrable, regardless of the nature of the law of $X$.

Thus, continuity is a requirement on the coordinate chart $G$, not on the random variable $X$.
\end{remark}
\begin{remark}\label{rem:general-chart-not-canonical}
Definition~\ref{def:prob-chart} yields a family of coordinate-induced barycenters, one for each admissible chart $G$. In general, such a chart is \emph{not} assumed to coincide with the distribution function of $X$. Consequently, for a general probability coordinate chart $G$, the transformed variable $G(X)$ need not be uniformly distributed on $(0,1)$. Thus the construction
\[
b_G(X)=G^{-1}\!\bigl(\mathbb{E}(G(X))\bigr)
\]
should be understood, at this level, as a chart-dependent barycentric operation in probability coordinates, rather than as an intrinsic object attached to the law of $X$.
\end{remark}

The canonical case arises when the chart is generated by the law of the random variable itself.

\begin{proposition}[Canonical chart and uniformization]\label{prop:canonical-chart-uniformization}
Let $X$ be a real-valued random variable with continuous strictly increasing distribution function
\[
F_X\colon I\to (0,1)
\]
on an interval $I\subseteq \mathbb{R}$. Then:

\begin{enumerate}
\item $F_X$ is a probability coordinate chart on $I$ in the sense of Definition~\ref{def:prob-chart};

\item the transformed variable
\[
U:=F_X(X)
\]
is uniformly distributed on $(0,1)$;

\item the associated barycenter satisfies
\[
b_{F_X}(X)=F_X^{-1}\!\bigl(\mathbb{E}(U)\bigr)
      =F_X^{-1}\!\left(\frac12\right).
\]
\end{enumerate}
\end{proposition}

\begin{proof}
Since $F_X$ is continuous and strictly increasing on $I$ and satisfies $F_X(I)=(0,1)$, it is a probability coordinate chart.

Set $U=F_X(X)$. For every $u\in(0,1)$,
\[
\mathbb{P}(U\le u)
=\mathbb{P}\bigl(F_X(X)\le u\bigr)
=\mathbb{P}\bigl(X\le F_X^{-1}(u)\bigr)
=F_X\!\bigl(F_X^{-1}(u)\bigr)
=u,
\]
so $U\sim \mathrm{Uniform}(0,1)$.

Therefore,
\[
\mathbb{E}(U)=\int_0^1 u\,du=\frac12,
\]
and hence
\[
b_{F_X}(X)
=F_X^{-1}\!\bigl(\mathbb{E}(F_X(X))\bigr)
=F_X^{-1}\!\bigl(\mathbb{E}(U)\bigr)
=F_X^{-1}\!\left(\frac12\right).
\]
\end{proof}

\begin{remark}[Probability coordinates beyond continuity]\label{rem:probability-coordinates-general}
The construction
\[
b_G(X)=G^{-1}\bigl(\mathbb{E}(G(X))\bigr)
\]
does not require $G$ to coincide with the distribution function of $X$, nor does it rely on continuity assumptions on the law of $X$: for any admissible chart $G$ and any law supported on $I$---continuous, discrete, or mixed---the quantity $\mathbb{E}(G(X))$ lies in $(0,1)$ and the barycenter is well defined. The essential feature of the construction is therefore not uniformization, which is specific to the continuous intrinsic case, but the use of probability itself as a coordinate system.

One combination lies outside the present framework and deserves explicit mention. If $X$ has atoms and one attempts the intrinsic choice $G=F_X$, then $F_X$ is discontinuous, fails Definition~\ref{def:prob-chart}, and its range omits the jump intervals: the classical inverse does not exist, and $\mathbb{E}(F_X(X))$ may fall inside a gap of the range attained by no value of $x$. In this case the map $x\mapsto F_X(x)$ still provides a useful parametrization of probability mass, in which atoms correspond to intervals of mass concentration, but no pullback $b_{F_X}(X)$ is claimed here. An extension of the intrinsic construction to atomic laws, based on the generalized quantile function $F_X^{-1}(u)=\inf\{x: F_X(x)\ge u\}$, is developed in companion work on the discrete and mixed case.
\end{remark}
\subsubsection{Existence of the Kolmogorov-type functional}

\begin{proposition}[Existence beyond integrability]
Let $G$ be a probability coordinate chart on $I$ and let $X$ be an arbitrary 
$I$-valued random variable. Then $G(X)$ is bounded and therefore integrable, 
and $\mathbb{E}(G(X))\in(0,1)$ lies in the domain of $G^{-1}$. In particular, 
the probability barycenter $b_G(X)$ of Definition~\ref{def:prob-barycenter} is 
always well-defined, regardless of whether $\mathbb{E}(X)$ exists.
\end{proposition}
\begin{remark}[Why this is not merely an arbitrary transform]\label{rem:not-just-transform}
The quantity
\[
b_G(X):=G^{-1}\bigl(\mathbb{E}(G(X))\bigr)
\]
has the formal shape of a classical Kolmogorov or quasi-arithmetic mean. However, the perspective adopted in this paper is different. The point is not the algebraic template itself, but the interpretation of \(G\) as a probability coordinate chart. In this view, the map \(x\mapsto G(x)\) places observations on a probability scale, and the operation \(\mathbb{E}(G(X))\) performs averaging in probability coordinates rather than in value space.

This distinction is essential. For a general monotone generator, the
transformed variable $g(X)$ is merely a reparametrization of values. By
contrast, here $G$ is required to be the restriction of a continuous
cumulative distribution function. The coordinate $G(x)$ therefore carries
distributional meaning: it encodes the location of $x$ relative to the
probability mass of the chart law. The resulting quantity $b_G(X)$ is
interpreted as a barycenter in probability coordinates, not simply as an
unconstrained transformed average.

In particular, the present framework does not rely on continuity of the law of \(X\), nor on the exact uniformization property \(G(X)\sim \mathrm{Uniform}(0,1)\), which is specific to the continuous self-induced case. Its central idea is more general: probability itself may serve as a coordinate system for averaging, and different admissible charts lead to different, but mathematically meaningful, notions of barycenter.
\end{remark}
\begin{remark}[Extension beyond moment assumptions]
The above construction provides a notion of central tendency for random
variables for which the classical mean may fail to exist. This includes
discrete and continuous heavy-tailed laws whose support is contained in the
domain of the chosen chart.
\end{remark}
\begin{remark}[Orientation invariance]\label{rem:orientation}
Let $G$ be a probability coordinate chart and define $\bar G = 1 - G$.
Then $\bar G$ is strictly monotone decreasing and
\[
b_{\bar G}(X) = b_G(X).
\]
Thus reversing the orientation of probability coordinates does not
affect the probability barycenter. Note that $\bar G$ is decreasing and
therefore lies outside the class of charts of
Definition~\ref{def:prob-chart}; the statement concerns the algebraic
template $G\mapsto G^{-1}(\mathbb{E}(G(X)))$ on a momentarily enlarged
class of generators.
\end{remark}
\begin{proposition}[Affine invariance of probability barycenters]
\label{prop:affine-invariance}
Let $G$ be a probability coordinate chart and let
\[
\tilde G(x) = a\,G(x) + b,
\quad a \neq 0,
\]
be an affine transformation such that $\tilde G$ remains strictly
monotone on the interval of invertibility.

Then, whenever both barycenters are well defined,
\[
b_{\tilde G}(X) = b_G(X).
\]
\end{proposition}

\begin{proof}
One has
\[
\E(\tilde G(X)) = a\,\E(G(X)) + b.
\]
Since $\tilde G^{-1}(u) = G^{-1}((u-b)/a)$, it follows that
\[
b_{\tilde G}(X)
= \tilde G^{-1}(\E(\tilde G(X)))
= G^{-1}(\E(G(X)))
= b_G(X).
\]
\end{proof}
\begin{remark}[Geometric interpretation]
The probability barycenter depends only on the affine structure of the
probability scale $(0,1)$ and is invariant under affine reparametrizations
of probability coordinates. This reflects the fact that barycenters are
defined through linear averaging in probability space. As in
Remark~\ref{rem:orientation}, the transformed generator $\tilde G=aG+b$ is
generally not a cumulative distribution function, so
Proposition~\ref{prop:affine-invariance} is a statement about the algebraic
template on an enlarged class of generators; within the class of charts
itself, the barycenter-preserving reparametrizations are characterized in
Proposition~\ref{prop:rigidity} below.
\end{remark}
\begin{remark}[Optimization and geometry]
\label{remark:optimization_vs_geometry}

Generalized means of the form
\[
h^{-1}\!\left(\frac{1}{n}\sum_{i=1}^n h(x_i)\right)
\]
admit a classical interpretation as least squares estimators after
transforming the data via a monotone function $h$; see
Berger and Casella~\cite{BergerCasella1992}.
This perspective is inherently variational, relying on the minimization of a
quadratic loss in transformed coordinates.

The probability barycenter admits the same variational description: by the
geometric characterization above, $b_G(X)$ minimizes
$c\mapsto\mathbb{E}(d_G(X,c)^2)$, which is precisely a quadratic loss in the
coordinates induced by $G$. The two readings therefore describe one and the
same object, not two competing constructions. What distinguishes the present
framework is not the absence of a variational principle, but what constrains
the generator: in the classical variational reading the transformation $h$ is
an arbitrary monotone function chosen alongside a loss, whereas here the
chart $G$ is required to be a cumulative distribution function, so that the
coordinates---and hence the loss---carry probabilistic meaning.

\end{remark}
\begin{proposition}[Optimization representation of generalized means]
\label{prop:optimization_representation}

Let $h:I \to \mathbb{R}$ be a continuous, strictly monotone function,
and let $x_1,\dots,x_n \in I$.
Then the generalized mean
\[
M_h(x_1,\dots,x_n)
= h^{-1}\!\left(\frac{1}{n}\sum_{i=1}^n h(x_i)\right)
\]
is the unique minimizer of the quadratic functional
\[
a \mapsto \sum_{i=1}^n \big(h(x_i) - h(a)\big)^2.
\]

\end{proposition}

\begin{proof}
The function $a \mapsto \sum_{i=1}^n (h(x_i)-h(a))^2$
is minimized when $h(a)$ equals the arithmetic mean of
$h(x_1),\dots,h(x_n)$, that is,
\[
h(a) = \frac{1}{n}\sum_{i=1}^n h(x_i).
\]
Since $h$ is strictly monotone, it is invertible, and the result follows.
\end{proof}

\begin{remark}[Two readings of one object]
\label{rem:two_readings}

The preceding results identify two equivalent descriptions of the same
notion of central tendency:

\begin{enumerate}
\item[(i)] (\emph{Variational reading})  
Given a monotone transformation $h$, the generalized mean
\[
M_h(x_1,\dots,x_n)
= h^{-1}\!\left(\frac{1}{n}\sum_{i=1}^n h(x_i)\right)
\]
is the minimizer of a quadratic loss in transformed coordinates
(Proposition~\ref{prop:optimization_representation}).

\item[(ii)] (\emph{Geometric reading})  
Given a cumulative distribution function $G$, the probability barycenter
\[
b_G(X) = G^{-1}\big(\mathbb{E}(G(X))\big)
\]
transports the distribution to probability space, averages linearly there,
and pulls back to value space; it is simultaneously the minimizer of
$\mathbb{E}(d_G(X,\cdot)^2)$.
\end{enumerate}

For $h=G$ the two constructions coincide. The contribution of the geometric
reading is thus not a new object but an explanation and a constraint: it
identifies the quadratic loss of (i) as the pullback of the flat geometry of
the probability scale, and it restricts the admissible generators to
cumulative distribution functions, for which the coordinates carry
distributional meaning. In particular, the choice of generator ceases to be
an unconstrained modelling degree of freedom and becomes the choice of a
probability geometry.
\end{remark}
\begin{proposition}[Rigidity of barycenter-preserving transformations]
\label{prop:rigidity}
Let $G$ be a probability coordinate chart, and let
\[
\widetilde G = T\circ G,
\]
where $T$ is continuous and strictly monotone on $(0,1)$.
Assume that for every admissible random variable $X$,
\[
b_{\widetilde G}(X)=b_G(X).
\]
Then $T$ is affine:
\[
T(u)=au+b,\qquad a\neq 0.
\]
\end{proposition}

\begin{proof}
Since $(T\circ G)^{-1}=G^{-1}\circ T^{-1}$, the identity
$b_{\widetilde G}(X)=b_G(X)$ is equivalent to
\[
T^{-1}(\E(T(U)))=\E(U)
\]
for all random variables $U$ taking values in $(0,1)$, i.e.
\[
\E(T(U))=T(\E(U)).
\]
Taking $U$ with two-point distributions yields
\[
T(pu+(1-p)v)=pT(u)+(1-p)T(v),
\]
for all $u,v\in(0,1)$ and $p\in[0,1]$; being continuous and both convex
and concave, $T$ is affine (see, e.g., \cite{Rockafellar1970}).
\end{proof}

\begin{corollary}[Rigidity within normalized probability charts]
\label{cor:chart-rigidity}
Let $G,H:I\to(0,1)$ be probability coordinate charts. If
\[
b_H(X)=b_G(X)
\]
for every $I$-valued random variable $X$, then $H=G$.
\end{corollary}

\begin{proof}
Apply Proposition~\ref{prop:rigidity} with $T=H\circ G^{-1}$. Since both $G$
and $H$ are increasing bijections onto $(0,1)$, the affine map $T$ must map
$(0,1)$ increasingly onto itself and satisfy the endpoint normalizations.
Hence $T(u)=u$ and $H=G$.
\end{proof}

\begin{remark}[Saturation--rigidity duality]
The chart dependence of the barycenter is therefore rigid rather than
accidental. Apart from a passive affine relabelling on an enlarged generator
class, two distinct normalized probability charts cannot define the same
barycenter for every law. Combined with
Theorem~\ref{thm:order-metric-phase-transition} at $p=2$, this produces a
useful duality. For a fixed non-degenerate $X$,
\[
\{b_G(X):G\text{ admissible}\}
=
(\operatorname*{ess\,inf}X,\operatorname*{ess\,sup}X),
\]
whereas equality of the barycentric functional on every admissible law implies
\[
b_G(Y)=b_H(Y)\quad\text{for every }Y
\qquad\Longrightarrow\qquad
G=H.
\]
Thus a single law allows maximal geometric freedom, whereas equality of the
entire barycentric functional imposes complete chart rigidity.
\end{remark}

\subsubsection{Kolmogorov equivalence}

\begin{definition}[Kolmogorov equivalence]
Let $G$ be a probability coordinate chart. Two random variables 
$X$ and $Y$ are said to be \emph{Kolmogorov-equivalent (with respect to $G$)} if
\[
\mathbb{E}(G(X)) = \mathbb{E}(G(Y)).
\]
\end{definition}

\begin{proposition}[Invariance of the Kolmogorov representative]
Let $X$ and $Y$ be Kolmogorov-equivalent with respect to a chart $G$. Then
\[
b_G(X) = b_G(Y).
\]
\end{proposition}

\begin{proof}
This follows immediately from the definition, since both quantities are obtained 
by applying $G^{-1}$ to the same value $\mathbb{E}(G(X)) = \mathbb{E}(G(Y))$.
\end{proof}

\begin{remark}[Interpretation]
Kolmogorov equivalence identifies random variables that share the same barycenter 
in probability coordinates. The functional $b_G$ thus selects a representative of 
each equivalence class by transporting the linear average in coordinate space 
back to the original domain.
\end{remark}

\subsubsection{The canonical case}

\begin{remark}[Self-generated coordinates]
The Kolmogorov equivalence is a structural feature of the construction and does 
not depend on the relation between $G$ and the law of $X$.

A distinguished situation arises when
\[
G = F_X,
\]
that is, when the coordinate chart is generated by the distribution of $X$ itself. 
In this case, the construction becomes canonical with respect to $X$, and the 
associated Kolmogorov representative coincides with the median.
\end{remark}

\subsubsection{On invertibility and coordinate patches}

Although a cumulative distribution function is monotone non-decreasing, it may 
fail to be strictly increasing on its entire domain. Consequently, the inverse 
$G^{-1}$ may not be globally well-defined.

To address this issue, we restrict $G$ to intervals on which it is strictly 
increasing. On such intervals, $G$ admits a well-defined inverse and the 
Kolmogorov-type construction can be carried out without ambiguity.

\begin{remark}[Geometric interpretation]
The restriction to regions of strict increase can be interpreted as the passage 
to admissible coordinate patches. Flat regions of the CDF correspond to 
singularities of the coordinate system, where local inversion is not possible.
\end{remark}

\subsection{Relation to Kolmogorov-type transformations}\label{sec:K-relation}

The construction of probability coordinates is closely related to the class of
monotone transformations that appear in the theory of Kolmogorov
means~\cite{Kolmogorov1930,Nagumo1930,Aczel1966}. However, the role of this connection is interpretative rather than definitional:
the Kolmogorov form arises here as a consequence of averaging in
probability coordinates, rather than as a starting point. Recall that
Kolmogorov means were originally defined as aggregation operators
for finite samples. Given a strictly monotone function $g$, the
Kolmogorov mean of real numbers $x_{1},\ldots,x_{n}$ is
\[
M_{g}(x_{1},\ldots,x_{n}) = g^{-1}\!\left( \frac{1}{n}\sum_{i=1}^{n} g(x_{i}) \right).
\]
When the argument is a random variable $X$, the corresponding
population-level functional
\[
b_g(X) = g^{-1}\!\bigl(\E[g(X)]\bigr)
\]
is no longer a mean in the original Kolmogorov sense, but an
expectation-like operator. This interpretation was articulated explicitly
by de~Carvalho~\cite{deCarvalho2016}. From an algebraic point of view,
the probability barycenter defined in Section~\ref{sec:barycenters}
coincides with a Kolmogorov expected value generated by the function $G$.
This formal similarity, however, obscures a fundamental conceptual
distinction.

In the classical theory of generalized means, the generating function $g$
is introduced abstractly and is typically chosen to satisfy axiomatic or
functional properties, such as continuity, monotonicity, or invariance
under aggregation. The choice of $g$ is largely independent of the
probabilistic structure of the random variable being averaged. As a
result, Kolmogorov means are often interpreted as alternative
averaging rules competing with classical expectation.

In contrast, in the present framework the generating function is not
arbitrary. Probability coordinate charts are required to be cumulative
distribution functions, and thus encode the distribution of probability
mass itself. The mapping $x \mapsto G(x)$ is not merely a change of scale. Once the
Euclidean metric of $(0,1)$ is pulled back, it generates the distance $d_G$ on
value space. Replacing $G$ by another chart while holding $X$ fixed therefore
changes the metric under which centrality is evaluated, rather than merely
rewriting one fixed metric in new coordinates; see
Section~\ref{sec:chart-generated-geometry}.

From this perspective, Kolmogorov means do not represent
competing notions of expectation. Rather, they arise as
\emph{coordinate-induced barycenters obtained by averaging in probability coordinates}. The Kolmogorov expectation associated with a chart $G$ is the pullback,
under $G^{-1}$, of the Euclidean barycenter of the probability
coordinates $G(X)$. The existence of such expectations follows from the
boundedness of probability coordinates (whose closure is compact), rather
than from ad hoc choices of generating functions. Stability after pulling
the coordinate average back to value space additionally depends on the
local regularity of $G^{-1}$.

This interpretation provides a conceptual explanation for earlier results
on the existence of Kolmogorov expectations under bounded generators.
Such generators are those that induce probability geometries in
which averaging is well defined. The present framework isolates this
geometric content and places it at the foundation of the theory.

\subsection{Intrinsic and benchmark geometries}\label{sec:intrinsic-benchmark}

The choice of a probability coordinate chart determines the geometry in
which centrality is defined. Two cases play a distinguished role
throughout the paper.

\paragraph{Intrinsic probability geometry.}
When the probability coordinate chart is chosen as the cumulative
distribution function of $X$, that is $G = F_{X}$, the resulting
probability geometry is intrinsic to the distribution. In this case, the
probability coordinates $F_{X}(X)$ are uniformly distributed on $(0,1)$,
and the geometry depends only on the internal structure of the
distribution rather than on an external reference. As shown in
Section~\ref{sec:barycenters}, the corresponding probability barycenter
coincides with the median of $X$. In this sense, the median arises as
the expectation associated with the intrinsic probability geometry of a
distribution.

\paragraph{Benchmark probability geometry.}
More generally, one may choose a probability coordinate chart $G$
independently of the distribution of $X$. Typical choices include
Gaussian, logistic, or other reference distribution functions. Such a
choice induces a benchmark geometry relative to which centrality and
dispersion are measured. The resulting probability barycenter quantifies
the location of $X$ relative to the benchmark probability structure
encoded by $G$.

The distinction between intrinsic and benchmark geometries is not a choice
between alternative coordinate descriptions of one fixed metric. Intrinsic
geometry is generated by the law of $X$ itself, whereas benchmark geometry
introduces an external probabilistic ruler. In both cases the law of $X$ is
kept fixed, but different charts induce different metrics on its support and
therefore different notions of centrality. This is precisely the active
chart-replacement operation of Proposition~\ref{prop:passive-active}, not the
passive covariance operation of Proposition~\ref{prop:passive-covariance}.

This flexibility is a defining feature of probability geometry. Rather
than imposing a single, canonical geometry on value space---as classical
expectation does via the Euclidean metric---probability geometry allows
the geometry to be chosen explicitly through probability coordinates.
Asymmetry, extremal phenomena, and tail mismatch with a benchmark are then
expressed as geometric features relative to the chosen coordinate system.

\begin{figure}[H]
\centering
\resizebox{\textwidth}{!}{%
\begin{tikzpicture}[
    >=Latex,
    every node/.style={font=\small},
    axis/.style={->, thick},
    map/.style={->, thick},
    point/.style={circle, fill=black, inner sep=1.6pt},
    chartone/.style={blue!70!black, thick},
    charttwo/.style={purple!80!black, thick},
    chartthree/.style={orange!85!black, thick}
]

\node[font=\bfseries\large] at (0,4.2) {Value space};
\node at (0,3.8) {one fixed random variable \(X\)};

\draw[axis] (-3.6,0) -- (3.6,0) node[right] {$x$};

\draw[black, thick, smooth, domain=-3.2:3.2, samples=150]
plot (\x,{2.25/(1+2.5*\x*\x)});

\node at (-2.45,1.15) {$f_X(x)$};

\foreach \x in {-3.0,-2.7,-2.4,-2.0,-1.5,-0.9,-0.4,-0.15,0.15,0.45,0.9,1.4,1.9,2.35,2.75,3.05}
    \node[circle, fill=black, inner sep=1pt] at (\x,0.06) {};

\node[align=center] at (0,-0.85)
{Same distribution,\\ same observations.};

\node[point, fill=blue!70!black] (egone) at (-0.75,-1.55) {};
\node[point, fill=purple!80!black] (egtwo) at (0.30,-1.55) {};
\node[point, fill=orange!85!black] (egthree) at (1.25,-1.55) {};

\node[below] at (-1.2,-1.79) {$b_{G_1}(X)$};
\node[below] at (0,-1.79) {$b_{G_2}(X)$};
\node[below] at (1.25,-1.79) {$b_{G_3}(X)$};

\draw[dashed, blue!70!black] (-0.75,-1.55) -- (-0.75,0);
\draw[dashed, purple!80!black] (0.30,-1.55) -- (0.30,0);
\draw[dashed, orange!85!black] (1.25,-1.55) -- (1.25,0);

\node[align=center] at (0,-2.85)
{Different benchmark charts produce\\
different probability barycenters in value space.};

\node[font=\bfseries\large] at (8,4.2) {Benchmark probability charts};
\node at (8,3.8) {different choices of \(G\)};

\draw[thick] (5.2,2.35) -- (10.8,2.35);
\draw[thick] (5.2,1.05) -- (10.8,1.05);
\draw[thick] (5.2,-0.25) -- (10.8,-0.25);

\node[left] at (5.2,2.35) {$0$};
\node[right] at (10.8,2.35) {$1$};
\node[left] at (5.2,1.05) {$0$};
\node[right] at (10.8,1.05) {$1$};
\node[left] at (5.2,-0.25) {$0$};
\node[right] at (10.8,-0.25) {$1$};

\node[point, fill=blue!70!black] (uone) at (7.05,2.35) {};
\node[point, fill=purple!80!black] (utwo) at (8.05,1.05) {};
\node[point, fill=orange!85!black] (uthree) at (8.95,-0.25) {};

\node[above, blue!70!black] at (7.05,2.48) {$\bar u_1=\mathbb{E}(G_1(X))$};
\node[above, purple!80!black] at (8.05,1.18) {$\bar u_2=\mathbb{E}(G_2(X))$};
\node[above, orange!85!black] at (8.95,-0.12) {$\bar u_3=\mathbb{E}(G_3(X))$};

\foreach \u in {5.35,5.55,5.85,6.20,6.70,7.15,7.55,8.10,8.75,9.25,9.80,10.30}
    \node[circle, fill=blue!70!black, inner sep=0.9pt] at (\u,2.43) {};

\foreach \u in {5.45,5.85,6.15,6.65,7.25,7.85,8.20,8.55,9.00,9.55,10.05,10.45}
    \node[circle, fill=purple!80!black, inner sep=0.9pt] at (\u,1.13) {};

\foreach \u in {5.30,5.75,6.35,6.95,7.55,8.10,8.75,9.20,9.65,10.10,10.45}
    \node[circle, fill=orange!85!black, inner sep=0.9pt] at (\u,-0.17) {};

\node[align=center] at (9,-2.65)
{Each chart induces its own coordinate representation,\\
coordinate barycenter, and pullback.};

\draw[map, blue!70!black] (2.7,1.25) .. controls (4.0,2.4) and (4.8,2.75) .. 
node[above] {$G_1$} (5.2,2.45);

\draw[map, purple!80!black] (2.7,0.85) .. controls (4.0,1.5) and (4.8,1.25) .. 
node[above] {$G_2$} (5.2,1.15);

\draw[map, orange!85!black] (2.7,0.45) .. controls (4.0,0.3) and (4.8,0.05) .. 
node[above] {$G_3$} (5.2,-0.15);

\draw[map, blue!70!black] (7.05,2.15) .. controls (5.4,-2.1) and (2.1,-2.0) ..
node[below] {$G_1^{-1}$} (egone);

\draw[map, purple!80!black] (8.05,0.85) .. controls (5.6,-2.45) and (2.2,-2.25) ..
node[below] {$G_2^{-1}$} (egtwo);

\draw[map, orange!85!black] (8.95,-0.45) .. controls (6.2,-2.75) and (2.6,-2.55) ..
node[below] {$G_3^{-1}$} (egthree);

\end{tikzpicture}
}

\caption{One fixed law under different chart-generated geometries. The
random variable $X$, its observations, and its law are held fixed, while the
benchmark chart is changed from $G_1$ to $G_2$ or $G_3$. Each chart resets the
Euclidean metric on its own probability scale, inducing a different metric
$d_{G_j}$ on value space, a different coordinate barycenter
$\bar u_j=\mathbb E[G_j(X)]$, and a different Fr\'echet centre
$b_{G_j}(X)=G_j^{-1}(\bar u_j)$. The differences are therefore geometric and
not merely alternative coordinate descriptions of one fixed centre.}
\label{fig:benchmark-coordinate-charts}
\end{figure}
\begin{remark}[Chart-indexed centrality]
The probability barycenter depends on the chart because it is the Fr\'echet
centre for the chart-generated metric $d_G$. Different admissible charts may
therefore produce different barycenters for the same fixed random variable
$X$. This is not a failure of coordinate invariance: coordinate invariance is
expressed by Proposition~\ref{prop:passive-covariance}, where the data are
transported together with the chart. Here the data remain fixed while the
metric changes.
\end{remark}
\begin{remark}[Comparison with escort distributions] Escort distributions, introduced in the context of generalized entropy frameworks 
(see, e.g., Tsallis~\cite{Tsallis1988} and Plastino and Plastino~\cite{Plastino1993}), 
modify a given probability density by reweighting its mass according to a nonlinear transformation, typically of the form
\[
f_q(x) = \frac{f(x)^q}{\int f(x)^q \, dx}.
\]
This construction alters the underlying probability measure, changing the relative importance of observations.

In contrast, the present framework does not modify the probability measure. 
Instead, it introduces a change of representation through probability coordinate charts $G$, mapping observations into the unit interval via $G(X)$. 
Averaging is then performed in probability coordinates and pulled back to value space.

Thus, escort distributions and probability-coordinate barycenters act on different structural levels:
escort distributions deform the measure, while probability coordinate charts deform the geometry.

In particular, while escort distributions suppress or amplify tail contributions through reweighting,
probability coordinates represent extreme observations near the boundary of $(0,1)$.
This bounded representation limits their direct contribution to coordinate
averages; robustness of the pullback barycenter requires additional regularity
of the inverse chart and is not asserted here solely from boundedness.
A more systematic comparison between geometric deformation (via probability coordinate charts) and measure deformation (via escort distributions), as well as their potential interactions, will be developed in subsequent work.
\end{remark}
\subsection{Outlook}

The next section consolidates the barycentric construction introduced above
and develops its probabilistic consequences. For every law supported on the
domain of an admissible chart, the coordinate expectation exists without a
moment assumption on $X$ and can be pulled back to value space.

\section{Probability barycenters}\label{sec:barycenters}
Section~\ref{subsec:order-metric-phase-transition} placed the quadratic
barycenter inside the full Fr\'echet family and showed that $p=2$ lies on the
maximally chart-sensitive side of the order--metric phase transition. We now
return to this quadratic case because its linear coordinate representation
supports the limit theory and Kolmogorov-moment constructions developed in the
remainder of the paper.

Let $G$ be a probability coordinate chart and let $I \subset \mathbb{R}$
be an interval on which $G$ is continuous and strictly increasing with
$G(I) = (0,1)$.

\subsection{Barycenters in probability coordinates}

Let $X$ be a real-valued random variable such that $\Prob(X \in I)=1$.

Since $G(X)$ takes values in $(0,1)$, it is bounded and therefore
integrable. This allows us to define the barycenter of $X$ in
probability coordinates as
\[
\mathbb{E}(G(X)).
\]

\subsubsection{Pullback to the original space}

Recall from Definition~\ref{def:prob-barycenter} that the
\textbf{probability barycenter} of $X$ with respect to the chart $G$ is
\begin{equation}\label{eq:barycenter}
b_G(X) = G^{-1}\!\bigl( \mathbb{E}(G(X)) \bigr),
\end{equation}
which is always well-defined since $\mathbb{E}(G(X))\in(0,1)$.

This construction consists of three steps:
\begin{enumerate}
\item mapping $X$ into probability coordinates via $G$;
\item averaging in the bounded space $(0,1)$;
\item transporting the result back to the original space via $G^{-1}$.
\end{enumerate}

\subsubsection{Existence without moment assumptions}

The definition of $b_G(X)$ does not require the existence of the
classical mean $\mathbb{E}(X)$. Indeed, the quantity $\mathbb{E}(G(X))$
is always well-defined since $G(X)$ is bounded.

This shows that the probability barycenter provides a notion of central
tendency for random variables with heavy tails, for which classical
moment-based definitions may fail.

\subsubsection{Relation to Kolmogorov equivalence}

Let $X$ and $Y$ be two random variables such that $\Prob(X \in I)=1$ and
$\Prob(Y \in I)=1$. If
\[
\mathbb{E}(G(X)) = \mathbb{E}(G(Y)),
\]
then
\[
b_G(X) = b_G(Y).
\]

Thus, the probability barycenter selects a representative of
each Kolmogorov equivalence class.

\subsubsection{The canonical case}

A distinguished case arises when $G = F_X$ and $F_X$ is continuous. As
established in Proposition~\ref{prop:canonical-chart-uniformization},
$F_X(X)\sim\Unif(0,1)$, so that
\[
b_{F_X}(X) = F_X^{-1}\left(\frac{1}{2}\right),
\]
which is the median of $X$: the intrinsic probability barycenter is the
median. Corollary~\ref{cor:intrinsic-frechet-collapse} shows that this is not a
peculiarity of the quadratic loss: under the intrinsic chart, every finite
Fr\'echet $p$-centre has the same pullback.

\subsubsection{Interpretation}

The probability barycenter $b_G(X)$ can be interpreted as a barycenter obtained by averaging in probability coordinates induced by the chart $G$. Unlike classical expectation,
which averages in value space, this construction performs averaging in
probability space, where all laws supported on the chart domain are
represented on a common bounded interval.

This viewpoint is particularly well-suited for heavy-tailed
distributions, where classical Euclidean averaging may fail to capture
the intrinsic structure of the data.

\medskip

\noindent
A natural question concerns the role and selection of the coordinate chart $G$.

\paragraph{On the choice of the probability coordinate chart.}
A central aspect of the framework is the choice of the cumulative distribution function $G$ acting as a probability coordinate chart. This choice is not canonical in general and should be regarded as a modeling decision that determines how probability mass is represented through coordinates.

In practice, the selection of $G$ may be guided by structural features of the distribution of $X$, such as symmetry, tail behavior, or regions of interest. In the distinguished case $G = F_X$, the distribution function of $X$, the representation becomes intrinsic: the random variable is mapped to the uniform distribution on $(0,1)$, and the associated barycenter can be interpreted as a canonical probability-coordinate representative. This holds only for continuous random variables; the discrete and mixed case requires a generalized-quantile extension of the intrinsic chart and is developed in companion work.

More generally, different choices of $G$ generate different metrics on the
fixed support of $X$. The resulting barycenters should therefore be understood
as chart-dependent Fr\'echet centres: each is absolute within its selected
probability geometry, while comparison across charts is a comparison across
geometries rather than across notations.
\section{Limit theorems in probability geometry}\label{sec:limit-thms}
In this section we study the asymptotic behavior of probability
coordinates and of the associated probability barycenters introduced
in Section~\ref{sec:barycenters}. The key observation is that classical limit
theorems operate naturally in probability space, independently of the
tail behavior of the underlying random variable. Two clarifications frame
what follows. First, the limit theorems below concern the transformed
functionals: they do not alter the behavior of the sample mean of $X$.
For example, when $\E|X|=\infty$ the standard strong law for the sample
mean is unavailable, and for a Cauchy law the sample mean has the same
Cauchy distribution for every sample size and therefore does not converge
to a finite deterministic centre.
Second, their content for bounded variables is classical; what the
framework adds is the identification of the limiting objects as
probability barycenters and the geometric reading of the assumptions.
The failure of classical expectation in heavy-tailed settings then appears
as a property of averaging in value space, rather than an obstruction to
probabilistic convergence for suitably represented functionals.
\subsection{Law of large numbers in probability coordinates}

Let $G$ be a probability coordinate chart and let $I \subset \mathbb{R}$
be an interval on which $G$ is continuous and strictly increasing with
$G(I)=(0,1)$.

Let $(X_n)_{n\geq1}$ be a sequence of independent and identically
distributed random variables such that $\Prob(X_1 \in I)=1$.

Define
\[
U_n := G(X_n).
\]

Since $U_n \in (0,1)$ almost surely, the variables $(U_n)$ are bounded
and therefore integrable.

\begin{theorem}[Law of large numbers in probability coordinates]
One has
\[
\frac{1}{n}\sum_{i=1}^n U_i \;\xrightarrow{a.s.}\; \mathbb{E}(U_1)
= \mathbb{E}(G(X_1)).
\]
\end{theorem}

\begin{proof}
This follows from the classical strong law of large numbers applied to
the bounded random variables $U_i = G(X_i)$; see, e.g., \cite{Billingsley1995,Kallenberg2002}.
\end{proof}
This convergence in probability coordinates will be transported to value space in the next result.

\subsection{Convergence of barycenters}
\begin{theorem}[Convergence of probability barycenters]
Under the assumptions above, suppose that $G^{-1}$ is continuous at
$\mathbb{E}(G(X_1))$. Then
\[
G^{-1}\left(\frac{1}{n}\sum_{i=1}^n G(X_i)\right)
\;\xrightarrow{a.s.}\;
b_G(X_1).
\]
\end{theorem}

\begin{proof}
By the law of large numbers,
\[
\frac{1}{n}\sum_{i=1}^n G(X_i) \to \mathbb{E}(G(X_1))
\quad \text{almost surely}.
\]
The conclusion follows by continuity of $G^{-1}$.
\end{proof}

\subsection{Empirical probability barycenters}

Let $X_{1},X_{2},\ldots$ be independent and identically distributed
real-valued random variables, and let $G$ be a probability coordinate
chart for $X_{1}$. Define the empirical probability barycenter by
\begin{equation}\label{eq:empirical-barycenter}
\widehat{b}_{G,n} \;=\; G^{-1}\!\left( \frac{1}{n}\sum_{i=1}^{n} G(X_{i}) \right).
\end{equation}

This quantity is the pullback, under $G^{-1}$, of the empirical
barycenter of the probability coordinates $G(X_{i})$.

\subsection{Central limit theorem for probability barycenters}
Second-order asymptotics are equally transparent in probability geometry.

\begin{assumption}[Local regularity of the coordinate chart]\label{ass:regularity}
Assume that the probability coordinate chart $G$ is differentiable at
$b_{G}(X_{1})$ with derivative $G'(b_{G}(X_{1})) \neq 0$.
\end{assumption}
The following result is a specialization, to probability coordinate charts,
of the central limit theorem for Kolmogorov-type transformations; see,
for example, \cite{deCarvalho2016}.

\begin{theorem}[Central limit theorem for probability barycenters]\label{thm:CLT}
Under Assumption~\ref{ass:regularity},
\begin{equation}\label{eq:CLT}
\sqrt{n}\left( \widehat{b}_{G,n} - b_{G}(X_{1}) \right)
\;\overset{d}{\longrightarrow}\;
\mathcal{N}\!\left( 0,\; \frac{\Var(G(X_{1}))}{(G'(b_{G}(X_{1})))^{2}} \right).
\end{equation}
\end{theorem}

\begin{proof}
By the classical central limit theorem for bounded random variables,
\[
\sqrt{n}\left( \frac{1}{n}\sum_{i=1}^{n} G(X_{i}) - \E(G(X_{1})) \right)
\;\overset{d}{\longrightarrow}\;
\mathcal{N}\!\bigl( 0,\,\Var(G(X_{1})) \bigr).
\]
Applying the delta method with the differentiable function $G^{-1}$
yields the result.
Thus, asymptotic normality of probability barycenters follows directly
from classical Gaussian fluctuations in probability coordinates via a
nonlinear change of variables.
\end{proof}

\subsection{Universality of Gaussian fluctuations}

The limiting distribution in Theorem~\ref{thm:CLT} is Gaussian
regardless of the tail behavior of $X_{1}$. Heavy-tailed phenomena
influence the limit only through the bounded random variable $G(X_{1})$.
In particular, divergence of classical moments of $X_{1}$ plays no role
in the asymptotic distribution of empirical probability barycenters.

This universality reflects a geometric fact: fluctuations are governed by
local behavior of probability mass in probability space, not by extreme
magnitudes in value space.

\subsection{Intrinsic probability geometry}

The intrinsic probability geometry yields a particularly transparent form
of the central limit theorem.

\begin{corollary}[Intrinsic CLT for the empirical probability barycenter]
Suppose that $F_X$ is continuous and strictly increasing on its chart domain,
and has a density $f$ in a neighborhood of its median $m$, with $f(m)>0$. Then
\[
\sqrt{n}\bigl(\hat b_{F_X,n} - m\bigr)
\;\xrightarrow{d}\;
\mathcal{N}\!\left(0,\frac{1}{12\,f(m)^2}\right).
\]
\end{corollary}

\begin{proof}
Since $F_X(X_1) \sim \mathrm{Unif}(0,1)$, we have
\[
\mathrm{Var}(F_X(X_1)) = \frac{1}{12}.
\]
Moreover, $F_X'(m) = f(m)$. The result follows from Theorem~\ref{thm:CLT} by the delta method.
\end{proof}

\begin{remark}[Relation to the sample median; oracle character of the intrinsic estimator]
The estimator $\hat b_{F_X,n}$ differs from the classical sample median.
While both target the same population quantity $m = F_X^{-1}(1/2)$,
their constructions are different: the sample median is defined through order
statistics, whereas $\hat b_{F_X,n}$ is obtained by averaging in probability
space and transporting back via $F_X^{-1}$. Their asymptotic variances also
differ: the sample median has asymptotic variance $1/(4f(m)^2)$
\cite{Serfling1980,Reiss1989}, while $\hat b_{F_X,n}$ has asymptotic variance
$1/(12f(m)^2)$.

Two caveats are essential to interpret this comparison. First,
$\hat b_{F_X,n}$ is an \emph{oracle} quantity: it requires the true
distribution function $F_X$, which is unknown precisely in the situations of
interest, so the threefold variance reduction is a theoretical statement
about the intrinsic geometry, not an attainable statistical improvement.
Second, the natural feasible version degenerates: replacing $F_X$ by the
empirical distribution function $\widehat F_n$ yields, for continuous $F_X$,
\[
\frac{1}{n}\sum_{i=1}^n \widehat F_n(X_i)=\frac{n+1}{2n}
\quad\text{deterministically},
\]
by the rank identity, so the plug-in estimator is an order statistic of rank
approximately $n/2$ and recovers the sample median together with its
classical asymptotics. The statistical content of the framework therefore
resides in benchmark charts and estimated charts rather than in the
intrinsic case; a systematic treatment, including this degeneracy, is given
in the statistical companion paper.
\end{remark}
\subsection{Non-asymptotic concentration}

Beyond asymptotic statements, boundedness of probability coordinates yields
concentration inequalities with universal constants, involving no moment
conditions on $X$ and no constants depending on the chart. The natural
formulation is in the induced metric $d_G(x,y)=|G(x)-G(y)|$ of
Section~\ref{sec:prob-coords}.

\begin{proposition}[Concentration of empirical probability barycenters]
\label{prop:concentration}
Let $X_1,X_2,\ldots$ be independent and identically distributed with
$\Prob(X_1\in I)=1$, and let $G$ be a probability coordinate chart on $I$.
Then for every $n\ge 1$ and every $t>0$,
\[
\Prob\Bigl( d_G\bigl(\widehat b_{G,n},\, b_G(X_1)\bigr) \ge t \Bigr)
\;\le\; 2\exp\bigl(-2nt^2\bigr).
\]
\end{proposition}

\begin{proof}
Since $G$ is strictly increasing on $I$,
\[
d_G\bigl(\widehat b_{G,n}, b_G(X_1)\bigr)
=\Bigl|\,\frac1n\sum_{i=1}^n G(X_i)-\E(G(X_1))\Bigr|.
\]
The variables $G(X_i)$ are independent and take values in $[0,1]$, so
Hoeffding's inequality \cite{Hoeffding1963} applies.
\end{proof}

\begin{remark}
The inequality is expressed on the probability scale, where fluctuations are
measured intrinsically; a value-space statement follows on any region where
$G^{-1}$ is Lipschitz, with the Lipschitz constant entering the rate. This is
the natural division of labour in probability geometry: universal
concentration on the probability scale, and chart-dependent distortion under
the pullback.
\end{remark}

\subsection{Geometric diagnosis of classical failure}

The preceding subsections show that classical limit theorems apply without
obstruction to the bounded coordinates $G(X)$. When a value-space law lacks
the moment assumptions required by the usual law of large numbers or central
limit theorem, that failure concerns the original sample mean, not every
functional of the observations.

Euclidean value-space distances give extreme observations unbounded leverage
in moment calculations. Probability geometry instead uses bounded coordinates
with compact closure, so coordinate averages and their fluctuations are well
defined. This change defines different functionals; it does not repair the
original moments.

For example, the sample mean of a Cauchy law remains Cauchy-distributed at
every sample size and does not converge to a finite deterministic centre.
The geometric viewpoint supplies alternative, chart-dependent functionals
for which regularity holds. Divergence of classical moments remains a property
of the law measured in value-space geometry, while probability barycenters
average in a stated bounded chart.
\subsection{Geometric interpretation of limit theorems}

\begin{remark}
The law of large numbers and central limit theorem arise as Euclidean phenomena in probability coordinates applied to the transformed variables $G(X_i)$.
\end{remark}

\begin{remark}
The convergence of
\[
G^{-1}\!\left( \frac{1}{n} \sum_{i=1}^n G(X_i) \right)
\]
can therefore be interpreted as barycentric stabilization in the geometry induced by $G$.
\end{remark}

\begin{remark}
The asymptotic variance in value space is obtained by transporting fluctuations through the inverse chart $G^{-1}$, leading to a scaling governed by $G'$.
\end{remark}
\subsection{Dictionary between value space and probability space}

\begin{center}
\small
\begin{tabular}{c c c}
\textbf{Value space} & $\longleftrightarrow$ & \textbf{Probability space} \\
\hline
Large positive/negative values & $\longleftrightarrow$ & Proximity to $1/0$ \\
Tail mismatch with benchmark $G$ & $\longleftrightarrow$ & Excess boundary occupation \\
Nonexistent value-space moments & $\longleftrightarrow$ & Finite coordinate moments \\
Value-space extremes & $\longleftrightarrow$ & Chart-relative boundary points \\
\end{tabular}
\end{center}
\begin{remark}
Boundary occupation is meaningful only relative to a stated chart and a
reference law. For the intrinsic continuous chart $G=F_X$, $G(X)$ is uniform,
so it contains no marginal tail signature.
\end{remark}
\begin{remark}
Boundedness guarantees existence of coordinate moments and averages.
Interpretation as a tail diagnostic additionally requires a benchmark and
control of location and scale, since those mismatches can also move mass
toward the boundary.
\end{remark}
\subsection{Outlook}

The results of this section establish that probability barycenters admit
laws of large numbers and central limit theorems under minimal regularity
assumptions. These results hold universally, independently of tail
behavior, and reflect geometric properties of probability space rather
than integrability in value space.

In the next section we extend this geometric viewpoint beyond centrality
and introduce Kolmogorov moments as probability-space descriptors of
dispersion and benchmark-relative boundary behavior.

\section{Kolmogorov moments in probability geometry}\label{sec:Kolmogorov-moments}

Probability barycenters describe the geometric center of probability mass
in probability space. To obtain a fuller geometric description of a
distribution, however, one must also characterize how probability mass is
arranged relative to this center. In classical probability theory, this
role is played by moments and centred moments, which quantify dispersion
through magnitude in value space. Such quantities may fail to exist for
heavy-tailed distributions precisely because they depend on unbounded
geometry.

In probability geometry, dispersion admits a natural alternative
description. By applying averaging operations to transformations of probability
coordinates, one obtains a hierarchy of Kolmogorov moments that remain well defined under minimal assumptions and
encode distributional features geometrically rather than metrically.
Equivalently, Kolmogorov moments are functionals of the pushforward
distribution of $X$ under $G$, that is, of the random variable $G(X)$
taking values in $(0,1)$.

\subsection{Kolmogorov expectation revisited}

Let $X$ be a real-valued random variable and let $G$ be a probability
coordinate chart. The Kolmogorov expectation associated with $G$, in the
population-level sense of de~Carvalho~\cite{deCarvalho2016} rather than as a
finite-sample Kolmogorov--Nagumo mean, is
\begin{equation}\label{eq:Kolmogorov-exp-recall}
b_G(X) \;=\; G^{-1}\!\bigl( \E(G(X)) \bigr),
\end{equation}
the probability barycenter of Section~\ref{sec:barycenters}.

Geometrically, this construction corresponds to averaging in probability
space and transporting the result back to value space through the inverse
coordinate map. Kolmogorov moments arise by extending this averaging
operation to nonlinear functions of probability coordinates.

\subsection{Initial Kolmogorov moments}

\begin{definition}[Initial Kolmogorov moments]\label{def:initial-moments}
Let $r \geq 1$. The initial Kolmogorov moment of order $r$ associated
with the probability coordinate chart $G$ is defined by
\begin{equation}\label{eq:initial-moment}
M_{r}^{(G)}(X) \;=\; G^{-1}\!\left( \E\bigl((G(X))^{r}\bigr) \right),
\end{equation}
whenever the inverse $G^{-1}$ is defined at $\E((G(X))^{r})$.
\end{definition}

Since $G(X)$ is bounded, the expectation $\E((G(X))^{r})$ exists for all
$r \geq 1$. Consequently, initial Kolmogorov moments are well defined for
every random variable supported on the chart domain, independently of
classical moment conditions.
\begin{remark}[Moments in probability space vs value space]
The probability-space moment
\[
\E\bigl[(G(X))^{r}\bigr]
\]
is computed entirely in chart coordinates. Applying $G^{-1}$ transports this
probability level back to value space, yielding a representative in the
original domain.
\end{remark}

\subsection{Geometric interpretation}

Initial Kolmogorov moments are moments taken in probability space, not in
value space. If $U=G(X)$, the raw moment $\E(U^r)$ increasingly emphasizes
the \emph{upper} boundary $u=1$ as $r$ grows; it does not symmetrically probe
both ends of the interval. Lower-boundary behavior is described by
$\E((1-U)^r)$, while the two-sided coordinate descriptor
\begin{equation}\label{eq:symmetric-boundary-moment}
S_{r}^{(G)}(X):=\E\!\left(U^r+(1-U)^r\right)
\end{equation}
emphasizes both boundaries. These are chart-relative descriptors: their
interpretation depends on whether $G$ is intrinsic or a stated benchmark.

From this viewpoint, Kolmogorov moments encode dispersion geometrically.
Their finiteness is a consequence of bounded probability coordinates
rather than integrability of $X$. The Kolmogorov moments admit an alternative representation in terms of a generating function, obtained by lifting the classical moment generating function to probability coordinates.
\begin{remark}[Pseudo-generating function for Kolmogorov moments]
Let $X$ be supported on the domain of a probability coordinate chart $G$.
Define
\[
\varphi_G(t) := \mathbb{E}\big(e^{t G(X)}\big).
\]
Because $G(X)$ is bounded, $\varphi_G$ is finite for all $t\in\R$ and
\[
G^{-1}\big(\varphi_G'(0)\big) = b_G(X),
\]
and more generally,
\[
G^{-1}\big(\varphi_G^{(k)}(0)\big)=M_k^{(G)}(X)
\]
for each integer $k\ge1$.

Thus, Kolmogorov moments can be interpreted as classical moments computed in probability coordinates and subsequently pulled back to value space.

This observation shows that the present framework admits a generating-function viewpoint analogous to the classical moment generating function, but expressed in probability coordinates.
\end{remark}

\subsection{Intrinsic probability geometry}

When the probability coordinate chart is chosen intrinsically, that is
$G = F_{X}$ and $X$ has a continuous distribution, the probability
integral transform yields
\[
F_{X}(X) \sim \Unif(0,1).
\]
In this case,
\begin{equation}\label{eq:intrinsic-moments}
\E\bigl((F_{X}(X))^{r}\bigr) = \frac{1}{r+1},
\end{equation}
and intrinsic Kolmogorov moments reduce to quantile-type functionals
determined entirely by the uniform geometry of probability space.

\subsection{Centred Kolmogorov moments}

Dispersion relative to the coordinate barycenter is measured on the
probability scale.

\begin{definition}[Centred Kolmogorov moments]\label{def:centred-moments}
Let $r \geq 1$ and let $\bar u = \E(G(X))$. The centred Kolmogorov moment
of order $r$ associated with $G$ is defined by
\begin{equation}\label{eq:centred-moment}
\widetilde{M}_{r}^{(G)}(X)
\;=\; \E\bigl(\,(G(X) - \bar u)^{r}\,\bigr).
\end{equation}
Absolute centred Kolmogorov moments are defined analogously, with
$\lvert G(X)-\bar u\rvert^{r}$ in place of $(G(X)-\bar u)^{r}$; these are
useful in situations where symmetry is not assumed.
\end{definition}

These quantities measure dispersion in probability space relative to the
coordinate barycenter $\bar u$, in direct analogy with centred classical
moments. Unlike initial Kolmogorov moments, centred Kolmogorov moments are
deliberately \emph{not} pulled back to value space.

\begin{remark}[Why centred moments are not pulled back]\label{rem:no-pullback}
Two obstructions make a pullback of centred moments through $G^{-1}$
inappropriate. The first is a domain obstruction: centred moments of odd
order may vanish or be negative, while $G^{-1}$ is defined only on $(0,1)$;
for any law symmetric on the probability scale, the odd centred moments are
zero and a pullback would simply not exist. The second is semantic:
$G^{-1}$ interprets its argument as a \emph{probability level}, whereas a
centred moment is a \emph{dispersion}. In the intrinsic case, for instance,
$\Var(F_X(X))=1/12$, and $F_X^{-1}(1/12)$ is a low quantile of $X$---a
location, not a measure of spread. Centred Kolmogorov moments are therefore
reported on the probability scale, where they are dimensionally meaningful
and universally finite, with $0\le \widetilde M_2^{(G)}(X)\le 1/4$. Initial
Kolmogorov moments, by contrast, admit a meaningful pullback because
$\E((G(X))^r)\in(0,1)$ is itself a probability level.
\end{remark}

\subsection{Kolmogorov variance and dispersion intervals}

The centred Kolmogorov moment of order two defines the Kolmogorov
variance
\begin{equation}\label{eq:Kolmogorov-variance}
V_{G}(X) \;=\; \Var\bigl(G(X)\bigr) \;=\; \widetilde M_{2}^{(G)}(X)
\;\in\;\Bigl[0,\tfrac14\Bigr],
\end{equation}
a probability-scale measure of dispersion that is finite for every random
variable.

A value-space notion of dispersion consistent with the affine structure of
the probability scale (Proposition~\ref{prop:affine-invariance}) is obtained
by transporting an \emph{interval} rather than a number. With
$s=\sqrt{V_G(X)}$, define the \emph{Kolmogorov dispersion interval}
\begin{equation}\label{eq:dispersion-interval}
D_G(X) \;=\; \bigl[\,G^{-1}(\bar u - s),\; G^{-1}(\bar u + s)\,\bigr],
\end{equation}
whenever $[\bar u - s,\ \bar u + s]\subset(0,1)$. The interval $D_G(X)$
contains the probability barycenter $b_G(X)$ and describes the region of
value space corresponding to one standard deviation on the probability
scale; its asymmetry around $b_G(X)$ reflects the local distortion of the
inverse chart.
\begin{remark}[Universality]
Since $G(X)$ takes values in a bounded domain, all Kolmogorov moments
exist for every random variable supported on the chart domain. In particular, divergence of
classical moments of $X$ does not affect the existence of these
quantities, which are entirely determined by probability geometry.
\end{remark}
\subsection{Moment determinacy in probability coordinates}

Classical moment problems on unbounded domains may be indeterminate:
distinct heavy-tailed laws can share all finite moments, the lognormal
family being the standard example. In probability coordinates the situation
is the opposite.

\begin{theorem}[Kolmogorov moments determine the law]\label{thm:determinacy}
Let $G$ be a probability coordinate chart on $I$ and let $X$ and $Y$ be
$I$-valued random variables. If
\[
\E\bigl((G(X))^{r}\bigr) \;=\; \E\bigl((G(Y))^{r}\bigr)
\qquad\text{for all } r\ge 1,
\]
then $X$ and $Y$ have the same law. Equivalently, the sequence of initial
Kolmogorov moments $\{M_{r}^{(G)}(X)\}_{r\ge1}$ determines the law of $X$
on $I$.
\end{theorem}

\begin{proof}
The variables $G(X)$ and $G(Y)$ take values in $[0,1]$. A probability
distribution supported on a bounded interval is uniquely determined by its
moment sequence: the moment generating function is finite everywhere, and
distributions with a common, everywhere-finite moment generating function
coincide (see, e.g., \cite{Billingsley1995}). Hence $G(X)$ and $G(Y)$ are
equal in law. Since $G:I\to(0,1)$ is a continuous, strictly increasing
bijection, for every $x\in I$,
\[
\Prob(X\le x)=\Prob\bigl(G(X)\le G(x)\bigr)
=\Prob\bigl(G(Y)\le G(x)\bigr)=\Prob(Y\le x),
\]
so $X$ and $Y$ have the same law. The equivalence with the pulled-back
moments follows since $G^{-1}$ is injective, so the sequences
$\{M_r^{(G)}\}$ and $\{\E((G(X))^r)\}$ determine one another.
\end{proof}

\begin{remark}[Determinacy restored by the geometry]
For heavy-tailed laws, classical moments may fail to exist, and even when
they all exist they may fail to characterize the distribution.
Theorem~\ref{thm:determinacy} shows that both pathologies disappear in
probability coordinates: every law on $I$ possesses all Kolmogorov moments,
and these moments identify it uniquely. Determinacy is thus a structural
consequence of the compact closure $[0,1]$ of the coordinate space.
Boundedness itself guarantees existence of probability barycenters and
supports the concentration inequality of
Proposition~\ref{prop:concentration}.
\end{remark}

\subsection{Summary}

Kolmogorov moments provide probability-geometric descriptors of
dispersion that remain well defined for every law supported on the chart
domain, and, by
Theorem~\ref{thm:determinacy}, characterize the underlying law completely.
They complement probability barycenters by capturing how probability mass is
arranged relative to its geometric center, independently of classical
moment assumptions.

\section{Tail mismatch and boundary geometry}\label{sec:heavy-tails}

Heavy-tailedness is a property of a law in value space, conventionally
defined through tail decay, regular variation, or moment conditions; see
\cite{Resnick2007,Embrechts2003}. A probability chart does not change that
property. What it supplies is a bounded, chart-relative representation in
which mismatch with a benchmark can be measured near the boundary.

\subsection{Probability coordinates and tail events}

Let $X$ have law $F$ supported on the domain $I$ of a probability coordinate
chart $G$, and write $U=G(X)$. Since $G$ is increasing, observations far in
the upper or lower part of $I$ map near $1$ or $0$, respectively. For
$0<\varepsilon<1/2$, define the two-sided boundary occupation
\begin{equation}\label{eq:boundary-occupation}
B_{G,\varepsilon}(F)
:=\Prob_F\!\left(G(X)<\varepsilon\ \text{or}\ G(X)>1-\varepsilon\right).
\end{equation}
This quantity is always finite, but it is not an absolute measure of
heavy-tailedness: it depends on both $F$ and the chart $G$.

\subsection{Benchmark-relative boundary excess}

Let $F_0$ be a reference law on the same chart domain. The excess boundary
occupation of $F$ relative to $F_0$ is
\begin{equation}\label{eq:excess-boundary-occupation}
\Delta_{G,\varepsilon}(F;F_0)
:=B_{G,\varepsilon}(F)-B_{G,\varepsilon}(F_0).
\end{equation}
When $F_0$ is continuous and $G=F_0$, the probability integral transform
gives $B_{G,\varepsilon}(F_0)=2\varepsilon$. Positive excess then records
more mass near the chart boundaries than the benchmark assigns. It may be
caused by tail, location, or scale mismatch; consequently, a tail
interpretation requires location and scale to be controlled or separately
calibrated.

\subsection{One-sided and two-sided coordinate moments}

The raw moment $\E(U^r)$ emphasizes the upper boundary only. Its reflected
counterpart $\E((1-U)^r)$ emphasizes the lower boundary, and the symmetric
quantity $S_r^{(G)}(X)$ in \eqref{eq:symmetric-boundary-moment} emphasizes
both. Together with $B_{G,\varepsilon}(F)$, these bounded quantities describe
the pushforward law near the chart boundaries. They are geometric
descriptors of benchmark mismatch, not replacements for classical
definitions of tail decay.

\subsection{Intrinsic probability geometry}

If $F$ is continuous and the chart is intrinsic, $G=F$, then $U=F(X)$ is
uniform and
\[
B_{F,\varepsilon}(F)=2\varepsilon.
\]
Thus intrinsic marginal coordinates contain no signature that distinguishes
a heavy-tailed marginal from a light-tailed one. Marginal tail comparisons
require an external benchmark (or comparison between specified charts).
Heavy-tailedness itself remains an absolute property of the value-space law;
only its geometric representation is chart-relative.

\subsection{Geometric diagnosis of classical divergence}

Bounded probability coordinates guarantee the existence of coordinate
averages and coordinate moments even when corresponding value-space moments
do not exist. This is a change of functional and geometry, not a removal of
the original divergence. Pulling results back to value space may amplify
coordinate perturbations when $G^{-1}$ is steep, which is why the limit
theorems and stability statements above include local regularity conditions.

\subsection{Summary}

Boundary occupation provides a stable and explicit description of where a
pushforward law lies relative to a stated benchmark. Excess boundary
occupation can diagnose tail mismatch after location and scale have been
controlled, while intrinsic continuous coordinates are uniform and carry no
marginal tail signature.

\section{Multivariate probability geometry and copula coordinates}\label{sec:multivariate}

The geometric framework developed so far extends naturally to
multivariate random variables by working in probability coordinates
componentwise and encoding dependence through copulas. This extension
preserves the geometric interpretation of probability barycenters and
Kolmogorov moments while separating marginal geometry from dependence
geometry.

\subsection{Multivariate probability coordinates}

Let $X = (X_1, \ldots, X_d)$ be a random vector, and let 
$G = (G_1, \ldots, G_d)$ be a collection of probability coordinate charts.
Define the multivariate probability coordinates
\begin{equation}\label{eq:multivariate-coords}
U_i := G_i(X_i), \qquad i=1,\ldots,d.
\end{equation}
The random vector $U=(U_1,\ldots,U_d)$ takes values in $(0,1)^d$, whose
closure is the unit cube $[0,1]^d$. Its joint distribution defines a
probability measure on that cube. If all marginal CDFs are continuous and
qualify as charts on their respective domains, then in the intrinsic case
$G_i=F_{X_i}$ this distribution is the unique copula of $X$ by Sklar's
theorem; see \cite{Joe2014}.
In this representation, marginal behavior is encoded in the coordinate
maps $F_{X_{i}}$, while dependence structure is encoded entirely by the
copula.

\subsection{Probability geometry in \texorpdfstring{$[0,1]^{d}$}{[0,1]\^d}}

The geometry of multivariate probability space differs fundamentally from
that of value space. Regions near the corners of $[0,1]^{d}$ correspond
to joint tail events, while faces and edges correspond to partial tail
behavior.

Relative to a benchmark chart, joint tail mismatch or tail dependence can
therefore appear as excess occupation near the boundary of the unit cube.
For intrinsic continuous coordinates, each margin is uniform, so corner
behavior describes extremal dependence rather than marginal heaviness.

\subsection{Multivariate probability barycenters}

Let $I=I_1\times\cdots\times I_d$ and suppose $\Prob(X\in I)=1$. For each
$j$, let $G_j:I_j\to(0,1)$ be a probability coordinate chart, and define
$G:I\to(0,1)^d$ componentwise by
$G(x)=(G_1(x_1),\ldots,G_d(x_d))$.

The multivariate probability barycenter associated with $G$ is defined by
\begin{equation}\label{eq:multivariate-barycenter}
b_G(X) \;:=\; G^{-1}\!\bigl( \E(G(X)) \bigr),
\end{equation}
where $G^{-1}$ is also understood componentwise.

This definition generalizes the univariate probability barycenter and
coincides with it when $d=1$. Each component of $\E(G(X))$ lies in $(0,1)$,
so the inverse is defined. This componentwise restriction avoids a problem
with a general injective map into a nonconvex image: its coordinate mean
need not remain in that image.

Geometrically, $b_G(X)$ is the center of mass of the pushforward
distribution of $X$ in probability space, transported back to value space
through the inverse coordinate chart.

\begin{remark}[Componentwise character of the multivariate barycenter]
\label{rem:componentwise}
For componentwise charts $G=(G_1,\ldots,G_d)$, the barycenter factorizes:
\[
b_G(X)
=\bigl(G_1^{-1}(\E(G_1(X_1))),\ldots,G_d^{-1}(\E(G_d(X_d)))\bigr),
\]
that is, $b_G(X)$ is the vector of marginal probability barycenters; in the
intrinsic case $G_i=F_{X_i}$ it is the vector of marginal medians. The
dependence structure of $X$---equivalently, the copula of the coordinate
vector $U$---does not enter the barycenter itself. The copula
representation should therefore be understood as the geometric stage on
which genuinely joint functionals (corner mass, boundary occupation of
$[0,1]^d$, barycenters with respect to non-product metrics) can be
formulated; the systematic development of such joint descriptors is
deferred to future work.
\end{remark}

\begin{figure}[H]
\centering
\resizebox{0.92\textwidth}{!}{%
\begin{tikzpicture}[
    >=Latex,
    every node/.style={font=\small},
    point/.style={circle, fill=black, inner sep=1.2pt},
    bpoint/.style={circle, fill=purple!80!black, inner sep=2pt},
    map/.style={->, thick},
    axis/.style={->, thick}
]

\node[font=\bfseries\large] at (0,4.0) {Value space};
\node at (0,3.6) {$\mathbb{R}^3$};

\draw[axis] (0,0) -- (2.5,0) node[right] {$x_1$};
\draw[axis] (0,0) -- (0,2.3) node[above] {$x_2$};
\draw[axis] (0,0) -- (-1.6,-1.4) node[below left] {$x_3$};

\foreach \x/\y in {
0.25/0.20,0.55/0.35,0.85/0.55,1.10/0.75,1.35/0.95,
0.10/0.80,0.45/1.05,0.75/1.25,1.05/1.45,
-0.25/-0.30,-0.55/-0.55,-0.85/-0.80,
0.95/-0.15,1.35/-0.25,1.75/-0.35}
{
\node[point] at (\x,\y) {};
}

\node[bpoint] (EG) at (0.65,0.55) {};
\node[above right, purple!80!black] at (0.65,0.12)
{$b_G(X)$};

\node[align=center] at (0,-2.35)
{Random vector\\
$X=(X_1,X_2,X_3)$};

\node[font=\bfseries\large] at (7.2,4.0) {Probability coordinates};
\node at (7.2,3.6) {$[0,1]^3$};

\coordinate (A) at (5.4,-1.0);
\coordinate (B) at (8.6,-1.0);
\coordinate (C) at (8.6,1.7);
\coordinate (D) at (5.4,1.7);

\coordinate (shift) at (0.85,0.65);

\coordinate (A2) at ($(A)+(shift)$);
\coordinate (B2) at ($(B)+(shift)$);
\coordinate (C2) at ($(C)+(shift)$);
\coordinate (D2) at ($(D)+(shift)$);

\draw[thick] (A) -- (B) -- (C) -- (D) -- cycle;
\draw[thick] (A2) -- (B2) -- (C2) -- (D2) -- cycle;
\draw[thick] (A) -- (A2);
\draw[thick] (B) -- (B2);
\draw[thick] (C) -- (C2);
\draw[thick] (D) -- (D2);

\node[below] at (7.0,-1.25) {$u_1=G_1(X_1)$};
\node[rotate=90] at (5.05,0.35) {$u_2=G_2(X_2)$};
\node at (7.8,-0.15) {$u_3=G_3(X_3)$};

\node[below left] at (A) {$0$};
\node[below right] at (B) {$1$};
\node[above left, yshift=-2pt] at (D) {$1$};
\node[above right] at (B2) {$1$};

\foreach \x/\y in {
5.75/-0.55,6.05/-0.25,6.35/0.05,6.65/0.25,
6.95/0.55,7.25/0.75,7.55/0.95,7.85/1.10,
6.20/0.65,6.55/0.95,7.8/2.25,
7.40/0.10,7.75/0.65,8.10/0.70}
{
\node[point] at (\x,\y) {};
}

\node[bpoint] (ubar) at (6.8,0.35) {};
\node[above right, purple!80!black] at (6.9,0.12)
{$\bar u=\mathbb{E}(G(X))$};

\node[align=center] at (7.2,-2.35)
{Averaging is performed\\
inside the unit cube.};

\draw[map] (1.4,1.25) .. controls (3.8,2.25) and (4.65,2.25) ..
node[above, xshift=-0.35cm] {$G=(G_1,G_2,G_3)$} (7.35,1.55);

\draw[map] (ubar) .. controls (4.6,-3.05) and (2.2,-2.45) ..
node[pos=0.5, below] {$G^{-1}$}
(EG);
\end{tikzpicture}
}

\caption{Multivariate probability coordinates. A random vector \(X=(X_1,X_2,X_3)\) is mapped componentwise into \((0,1)^3\), whose closure is the unit cube, through probability coordinate charts \(G_1,G_2,G_3\). Averaging produces the coordinate barycenter \(\bar u=\mathbb{E}(G(X))\), which is then pulled back componentwise to value space.}
\label{fig:multivariate-probability-coordinates}
\end{figure}

\subsection{Intrinsic multivariate geometry}

When all marginal distribution functions are continuous and strictly
increasing on their respective chart domains, and $G$ is chosen as their
vector,
\begin{equation}\label{eq:intrinsic-multivariate}
G(x) = \bigl(F_{X_{1}}(x_{1}),\ldots,F_{X_{d}}(x_{d})\bigr),
\end{equation}
the probability coordinates $U$ have copula distribution $C$.

In this intrinsic geometry, the continuous marginals are uniformized and
probability geometry separates the marginal coordinate maps from dependence
structure, with the latter encoded by the copula. In view of
Remark~\ref{rem:componentwise}, the intrinsic multivariate barycenter
reduces to the vector of marginal medians: dependence enters not through
the barycenter itself, but through the geometry of the coordinate
distribution on $[0,1]^d$, which is the natural object for joint
descriptors.
\begin{remark}[Boundary geometry in higher dimensions]
In the multivariate setting, boundary regions of $[0,1]^d$ correspond to
joint tail events relative to the component charts. In intrinsic continuous
coordinates, concentration near corners reflects extremal dependence, while
faces or edges describe partial extremal events; it does not by itself
identify heavy marginal tails. Under benchmark charts, excess corner or face
occupation may also record marginal location, scale, or tail mismatch.
\end{remark}
\subsection{Summary}

The multivariate extension of probability geometry separates marginal
structure from dependence structure and represents both geometrically in
probability space. The multivariate barycenter itself is componentwise, a
vector of marginal barycenters; the geometry of the coordinate distribution
on $[0,1]^d$, with dependence encoded by the copula, provides the setting
in which genuinely joint descriptors of tail behavior can be developed.

This completes the probabilistic development of probability geometry in
both univariate and multivariate settings. Statistical constructions and
inferential methodologies based on these ideas are developed separately in
a dedicated companion paper.
\section{Discussion and outlook}\label{sec:discussion}

This paper develops a geometric framework for probability based on
probability coordinate charts and probability barycenters. By
interpreting cumulative distribution functions as coordinate maps and
working in probability coordinates rather than value space, expectation,
centrality, and dispersion are reinterpreted as geometric objects. In
this perspective, probability itself acquires a geometry, and expectation
becomes a barycentric operation in probability space.

The foundational distinction established in this paper is that a probability
chart is not merely a passive coordinate description of a pre-existing
geometry. In ordinary coordinate invariance, the metric tensor is transformed
along with the coordinate and geometric quantities remain unchanged. Here each
chart is assigned the Euclidean metric on $(0,1)$ and pulls it back to value
space. Replacing the chart while holding the observations and their law fixed
therefore replaces the metric-measure structure $(I,d_G,P_X)$. This explains,
in a single principle, why order-based summaries remain stable while metric
notions of centrality, dispersion, and boundary proximity are chart-sensitive.
It also clarifies that the variation of $b_G(X)$ across charts is the intended
geometric content of the framework, not an artefact to be eliminated.

The order--metric phase transition makes this distinction exact. For absolute
distance ($p=1$), the chart-generated Fr\'echet centre is the ordinary median
set and is therefore fixed by every admissible chart. For every strictly convex
power $p>1$, by contrast, the centre is metric to the maximal possible extent:
for a fixed non-degenerate law its orbit over admissible probability charts is
the whole open interval between the essential extremes. The transition is
already realized by a single analytic one-parameter deformation of any
baseline chart. In the intrinsic continuous geometry $G=F_X$, uniformization
restores symmetry and collapses all finite $p$-centres to the median. Thus the
framework does not merely assert that ordinal quantities are stable and metric
quantities may vary; it identifies a sharp boundary between a chart-rigid
Fr\'echet loss and a maximally chart-sensitive family.

A central insight of the framework is that failure of classical moments does
not preclude well-defined averages of bounded transforms. Classical
expectation averages in value space endowed with Euclidean geometry, where
extreme observations can have unbounded leverage. Probability charts instead
define different, bounded coordinate functionals. They do not make a
nonexistent classical moment exist; they provide a chart-indexed barycentric
structure whose assumptions and interpretation are explicit.

Within this geometric setting, probability barycenters exist for every law
supported on the domain of an admissible chart, independently of integrability
or moment assumptions on $X$. Laws of large numbers and central limit theorems hold in probability coordinates under the stated regularity assumptions, due to boundedness of the transformed variables, and Gaussian
fluctuations arise as a geometric consequence of averaging in probability
space. The intrinsic probability geometry associated with a distribution
yields the median as a canonical notion of centrality, while alternative
probability coordinate charts induce benchmark geometries relative to
which centrality is measured. These results are consistent with classical limit theorems for
Kolmogorov-type transformations, and here acquire a natural geometric
interpretation in probability space.

Kolmogorov moments further enrich this geometric picture by providing
chart-dependent probability-space descriptors of dispersion. Initial and
centred Kolmogorov moments remain finite for every law supported on the chart
domain. Raw powers emphasize one boundary, whereas reflected or symmetric
versions probe both. With a benchmark and a reference law, excess boundary
occupation describes tail mismatch after location and scale are controlled;
in the intrinsic continuous chart the coordinates are uniform and carry no
marginal tail signature.

For continuous marginals, the multivariate extension via intrinsic copula
coordinates separates marginal transforms from dependence structure.
Extremal dependence is represented by corner and face behavior of the copula,
whereas benchmark-coordinate boundary occupation can also reflect marginal
location, scale, or tail mismatch.

The framework developed here is deliberately probabilistic and
foundational. Detailed estimators and inferential procedures are reserved for
companion work. The paper instead identifies
geometric structures underlying probability that are universal across
distributions and independent of classical moment assumptions. This
separation clarifies the role of probability geometry as a foundational
layer upon which statistical methodology can be built. One structural
consequence deserves mention here. For a fixed differentiable benchmark
chart, if $b=b_G(F)$ and $0<|G'(b)|<\infty$, contamination differentiation
gives
\[
\operatorname{IF}(x;b_G,F)
=\frac{G(x)-\E_F[G(X)]}{G'(b)}.
\]
The numerator is bounded, so the influence function is bounded under these
local inverse-regularity conditions. Boundedness of the chart alone is not
sufficient if the derivative vanishes or the inverse is not differentiable.
This connection to robust statistics \cite{Huber2009}, together with feasible
estimation theory for benchmark and estimated charts, is developed in the
statistical companion paper.

The geometric interpretation developed here is related in spirit to recent work on quasi-arithmetic averages in information geometry \cite{Nielsen2023QAC}, where generalized averages arise from dual affine coordinate systems induced by convex potentials. However, the present framework differs fundamentally in that the geometry is not generated by arbitrary convex coordinates, but by probability-coordinate charts associated with cumulative distribution functions. This restriction preserves probabilistic interpretability and gives rise to a canonical intrinsic geometry through the probability integral transform. 

The present framework is also related, at a conceptual level, to optimal transport theory. Indeed, for a continuous random variable \(X\), the probability integral transform
\[
U = F_X(X)
\]
canonically transports the law of \(X\) onto the uniform distribution on \((0,1)\). In one dimension, this map coincides with the monotone transport associated with quantile representations and Wasserstein geometry; see, for example, \cite{Villani2009,Panaretos2019}. However, the goal of the present work is fundamentally different. Rather than studying transport costs or Wasserstein geodesics, the emphasis here is on the geometric role of probability-coordinate charts and the associated probability barycenters induced by averaging in transformed probability coordinates.

Several directions for further investigation follow naturally. The
moment-determinacy theorem (Theorem~\ref{thm:determinacy}) is a first
instance of the interaction between probability geometry and classical
moment problems on bounded domains; finer Hausdorff-type representations
and quantitative identification results remain to be explored. Boundary
geometry in probability space offers a natural
language for studying extremal dependence and rare-event phenomena. More broadly, probability geometry provides a unifying viewpoint for
understanding expectation, dispersion, and limit behavior across
distributions, including those with heavy tails. Its defining move is to make
the probabilistic ruler explicit: charts preserve the order of the data, but
their induced metrics determine what counts as distance, balance, and
extremeness. The geometry is therefore chosen rather than hidden.


\end{document}